\documentclass[preprint,12pt]{elsarticle}

\usepackage[bookmarks=true,
bookmarksnumbered=true, breaklinks=true,
pdfstartview=FitH, hyperfigures=false,
plainpages=false, naturalnames=true,
colorlinks=true,pagebackref=true,
pdfpagelabels]{hyperref}
\usepackage{amsfonts}
\usepackage{amssymb}
\usepackage{amsmath}
\usepackage{amsthm}
\usepackage{mathrsfs}

\newtheorem{thm}{Theorem}
\newtheorem{lem}{Lemma}
\newtheorem{pro}{Proposition}
\newtheorem{cor}{Corollary}

\journal{Journal of Functional Analysis}

\begin{document}

\begin{frontmatter}



\title{Generalized affine Hardy-Littlewood-Sobolev inequalities} 


\author{Youjiang Lin, Jinghong Zhou, Jiaming Lan} 


\begin{abstract}
We establishe an affine Hardy-Littlewood-Sobolev inequality concerning two different functions which
is stronger than the classical Hardy-Littlewood-Sobolev inequality. Furthermore, we also prove reverse inequalities for the new inequalities for log-concave functions.

\end{abstract}




\end{frontmatter}



Hardy–Littlewood–Sobolev inequality \cite{ELL01} states that for $p,r>1$, $0 < \alpha < n$, $\frac{1}{p}+\frac{1 }{r}-\frac{\alpha}{n}=1$, the following inequality holds for functions $f\in {L}^{p} (\mathbb{R}^n)$ and $h\in {L}^{r} (\mathbb{R}^n)$:
	\begin{eqnarray}\label{1a}
		C(n, \alpha, p){||f||}_{p}{||h||}_{r}\geq
		\left|\int_{\mathbb{R}^n}\int_{\mathbb{R}^n}\frac{f(x)h(y)}{{| x-y|}^{n-\alpha}}dxdy \right|,
	\end{eqnarray} 
	where $C(n, \alpha, p)$ is a sharp constant independent of $f$ and $h$ and satisfies
	\begin{eqnarray}\label{e1b}
		C(n, \alpha, p)\leq \frac{n}{\alpha}\omega _{n}^{1-\frac{\alpha}{n}}\frac{1}{pr}\left({\left(\frac{1-\frac{\alpha}{n}}{1-1/p}\right)}^{1-\frac{\alpha}{n}}+{\left(\frac{1-\frac{\alpha}{n}}{1-1/r}\right)}^{1-\frac{\alpha}{n}}\right),
	\end{eqnarray} 
	where $\omega _{n} $ represents the volume of the unit ball in $\mathbb{R}^n$.
	In the special case where $p=r=2n/(n+\alpha)$, the constant $C(n, \alpha, p)$ simplifies to:
	\begin{eqnarray*}
		C(n, \alpha, p)=C(n, \alpha)={\pi}^{(n-\alpha) /2}\frac{\Gamma (\alpha /2)}{\Gamma ((n+\alpha) /2)}{\left \{ \frac{\Gamma (n/2)}{\Gamma (n)}\right \}}^{-\alpha /n},
	\end{eqnarray*}
	where $\Gamma $ is the gamma function. 
	Equality in (\ref{1a}) holds if and only if $h= a_{0} f$ and $f(x)=a(b^{2}+|x-x_{0}|^{2})^{-\frac{n+\alpha }{2}}$ for some $a_{0}$, $a\in \mathbb{R}$, $0\neq b \in\mathbb{R}$ and $x_{0}\in \mathbb{R}^n$. 
	
	The Hardy–Littlewood-Sobolev inequality was initially introduced by G. Hardy and J. Littlewood in $\mathbb{R}^1$ \cite{HL28,HL30,HL32}, and was later extended by  S. Sobolev \cite{Sobolev63} to $\mathbb{R}^n$. However, neither of these inequalities is in its optimal form. Specifically, neither the best possible constant nor the extremal functions that would make the inequality (1.1) hold with this constant were identified. In a particular condition, E. Lieb \cite{Lieb83} provided the optimal version of (1.1) by proving the inequality with the best possible constant and identifying the functions that achieve equality with this constant. For the general case, neither the optimal constant nor the optimizers are known. For more detailed information on the Hardy–Littlewood–Sobolev inequality, we can refer to the monograph by E. Lieb and M. Loss \cite{ELL01}.
	 
	For geometric questions, affine inequalities are unchanged under translations and
	volume-preserving linear transformations. The most prominent example is the Petty projection inequality for convex bodies (i.e.,
	compact convex sets) in $\mathbb{R}^n$, which not only strengthens the Euclidean isoperimetric
	inequality but actually implies it (see, for example, \cite{BKJ13,CLYZ09,R. Gardner06,M.Ludwig2002,Schneider14}, and see \cite{AAKM04,CDDFH19,HS09,HaberlS09,JM2025,JM22-,ELDZ2000,ELDZ2002,Lin17,LX21,MilmanY,QV17,Nguyen16,T.Wang12,WX20,GZhang99} for related results in the ${L}_{p}$ Brunn–Minkowski theory and for more general sets). 
	
	J. Haddad and M. Ludwig \cite[Theorem 1]{JM22} established sharp affine Hardy-Littlewood-Sobolev inequalities for $0<\alpha <n$ and non-negative $f\in {L}^{2n/(n+\alpha)} (\mathbb{R}^n)$:
	\begin{eqnarray}\label{1b}
		\gamma _{n,\alpha }\left \| f \right \|_{\frac{2n}{n+  \alpha } }^{2}&\ge& n\omega _{n}^{\frac{n-\alpha }{n} }\left ( \frac{1}{n}\int_{\mathbb{S}^{n-1} }\left ( \int_{0}^{\infty }t^{\alpha -1}  \int_{\mathbb{R}^n }f\left ( x \right )f\left ( x+  t\xi  \right )dxdt  \right )^{\frac{n}{\alpha }}   d\xi\right )^{\frac{\alpha }{n} }\nonumber\\
		&\ge& \int_{\mathbb{R}^n } \int_{\mathbb{R}^n }\frac{f\left ( x \right )f\left ( y \right )  }{\left | x-y \right |^{n-\alpha }  }dxdy. 
	\end{eqnarray}
	 Equality holds in the first inequality if and only if the function $f$ has the form  $f\left( x \right )=a\left ( 1+\left | \phi\left( x-x_{0} \right )^{2}  \right|\right )^{-\frac{n+\alpha }{2} }$ for $x\in \mathbb{R}^n$, where $a\ge0$, $\phi\in GL\left ( n \right )$ and $x_{0}\in \mathbb{R}^n$. Here $GL(n)$ denotes general linear transformations.  For the second inequality, equality is attained precisely if $f$ is radially symmetric. 
	
	Here, the integral over the unit sphere $\mathbb{S}^{n-1}$ is taken with respect to the $(n-1)$-dimensional Hausdorff measure. The sharp affine Hardy-Littlewood-Sobolev inequalities (\ref{1b}) are significantly stronger than the classical Hardy-Littlewood-Sobolev inequality (\ref{1a}) for $f=h$.
	
	The primary goal of this paper focuses on proving affine Hardy-Littlewood-Sobolev inequalities concerning two different functions. 
	\begin{thm}\label{1.1}
		For $n\ge 1$, $1<p,r<\infty $, and $0<\alpha <n$ satisfying
		\begin{eqnarray}
			\frac{1}{p}+\frac{1}{r}-\frac{\alpha }{n}=1.\nonumber   
		\end{eqnarray}
		Then, there is a constant $C\left ( n,\alpha ,p \right )>0 $ such that for all non-negative functions $f\in L^{p}\left ( \mathbb{R}^n  \right ) $ and $h\in L^{r}\left ( \mathbb{R}^n  \right ) $,
		\begin{eqnarray}\label{e1.3}
			C\left ( n,\alpha ,p \right )\left \| f \right \|_{p}\left \| h \right \|_{r}&\ge& n\omega^{\frac{n-\alpha }{n} } _{n}\left ( \frac{1}{n}\int_{\mathbb{S}^{n-1} }\left ( \int_{0}^{\infty }t^{\alpha -1}\int_{\mathbb{R}^n }h\left ( x \right )f\left ( x+t\xi  \right )dxdt      \right )^{\frac{n}{\alpha } } d\xi     \right )^{\frac{\alpha }{n}}\nonumber\\
			&\ge&\int_{\mathbb{R}^n}\int_{\mathbb{R}^n}\frac{f\left ( x \right )h\left ( y \right )  }{\left | x-y \right |^{n-\alpha }  }dxdy.
		\end{eqnarray}
		If $p=r=\frac{2n}{n+\alpha } $, there is equality in the first inequality if and only if $h$ coincides with a translation of $a_{0} f$ where $a_{0}\in \mathbb{R}$ and
		\begin{eqnarray}
			f\left ( x \right )=a\left ( b ^{2}+\left | \phi \left ( x-x_{0} \right )^{2}    \right |   \right )^{-\frac{n+\alpha }{2} }\notag
		\end{eqnarray}
		for some $a\in \mathbb{R}$, $0\ne b \in \mathbb{R}$, $\phi \in GL\left ( n \right )$ and $x_{0}\in \mathbb{R }^{n}  $. There is equality in the second inequality if $f$, $h$ are radially symmetric.
	\end{thm}
	The generalized affine Hardy–Littlewood–Sobolev inequality (\ref{e1.3}) is stronger than Hardy–Littlewood–Sobolev inequality (\ref{1a}). When $p=r=\frac{2n}{n+\alpha } $ and $f=h$, we can get the affine Hardy-Littlewood-Sobolev inequalities (\ref{1b}).  
	
	To prove Theorem \ref{1.1}, for given $0< \alpha < n$, we define the star-shaped set $S_{\alpha }\left ( f,h \right )$ associated with the functions $f$ and $h$ by its radial function,
	\begin{eqnarray}\label{1c}
		\rho^{\alpha } _{S_{\alpha }\left ( f,h \right )  }\left ( \xi  \right )=\int_{0}^{\infty }t^{\alpha -1}\int_{\mathbb{R}^n }h\left ( x \right )f\left ( x+t\xi  \right )dxdt,\quad \xi \in \mathbb{R}^{n}\setminus\{0\}.  
	\end{eqnarray}  
	The first inequality presented in Theorem \ref{1.1} now can be expressed in the form 
	\begin{eqnarray}
		C\left ( n,\alpha ,p \right )\left \| f \right \|_{p}\left \| h \right \|_{r}\ge n\omega^{\frac{n-\alpha }{n} }_{n}\left | S_{\alpha  }\left ( f,h \right )   \right |^{\frac{\alpha }{n} } , \label{1d}
	\end{eqnarray}
	where $\left | \cdot  \right | $ denotes $n$-dimensional Lebesggue measure. Since both sides of \eqref{1d} remain invariant under same translations of $f$ and $h$, and that \begin{eqnarray}
		S_{\alpha }\left ( f\circ \phi ^{-1} ,h\circ \phi ^{-1} \right )=\phi S_{\alpha }\left ( f,h \right )  \notag  
	\end{eqnarray} 
	 holds for volume-preserving linear transformations $\phi :\mathbb{R}^n\to \mathbb{R}^n$, it follows that the inequality \eqref{1d} is affine.
	
	
	J. Dou and M. Zhu \cite{DZ15} and W. Beckner \cite{Beckner15} obtained sharp Hardy-Littlewood-Sobolev inequalities for $\alpha > n$ as follows:
	\begin{eqnarray}\label{e1.6}
		C\left ( n,\alpha ,p \right )\left \| f \right \|_{p}\left \| h \right \|_{r}\le\int_{\mathbb{R}^n}\int_{\mathbb{R}^n}\frac{f\left ( x \right )h\left ( y \right )  }{\left | x-y \right |^{n-\alpha }  }dxdy
	\end{eqnarray}
	for non-negative functions $f\in L^{p}\left ( \mathbb{R}^n  \right )$ and $h\in L^{r}\left ( \mathbb{R}^n  \right )$ satisfying $\frac{1}{p}+\frac{1}{r}-\frac{\alpha }{n}=1$, where $C\left ( n,\alpha ,p \right )$ is defined in (\ref{e1b}). 
	When $p=r=\frac{2n}{n+\alpha }$, equality in the inequality holds if and only if $h= a_{0} f$ and
	\begin{eqnarray}
		f\left ( x \right )=a\left ( b ^{2}+\left | \left ( x-x_{0} \right )^{2}    \right |   \right )^{-\frac{n+\alpha }{2} }\notag
	\end{eqnarray}
	for some $a_{0}$, $a\in \mathbb{R}$, $0\ne b \in \mathbb{R}$, and $x_{0}\in \mathbb{R }^{n}$.
	
	In addtion, J. Haddad and M. Ludwig \cite[Theorem 14]{JM22} etablished corresponding sharp affine Hardy-Littlewood-Sobolev inequalities for $\alpha > n$ and non-negative $f\in {L}^{2n/(n+\alpha)} (\mathbb{R}^n)$: \begin{eqnarray}\label{1bbb}
		\gamma _{n,\alpha }\left \| f \right \|_{\frac{2n}{n+  \alpha } }^{2}\le n\omega _{n}^{\frac{n-\alpha }{n} }\left | S_{\alpha  } f \right |^{\frac{\alpha }{n} }
		\le \int_{\mathbb{R}^n } \int_{\mathbb{R}^n }\frac{f\left ( x \right )f\left ( y \right )  }{\left | x-y \right |^{n-\alpha }  }dxdy. 
	\end{eqnarray}
     Equality holds in the first inequality if and only if the function $f$ has the form $f\left( x \right )=a\left ( 1+\left | \phi\left( x-x_{0} \right )^{2}  \right|\right )^{-\frac{n-\alpha }{2} }$ for $x\in \mathbb{R}^n$, where $a\ge0$, $\phi\in GL\left ( n \right )$ and $x_{0}\in \mathbb{R}^n$. For the second inequality, equality is attained precisely if $f$ is radially symmetric. 
	
	In Section \ref{5}, for $\alpha > n$,  we will establish affine HLS inequalities concerning two different functions.
	\begin{thm}\label{1.2}
		For $n\ge 1$, $0<p,r<1  $, and $\alpha >n$ satisfying
		\begin{eqnarray}
			\frac{1}{p}+\frac{1}{r}-\frac{\alpha }{n}=1,\nonumber   
		\end{eqnarray}
		there is a constant $C\left ( n,\alpha ,p \right )>0 $ such that for all non-negative $f\in L^{p}\left ( \mathbb{R}^n  \right ) $ and $h\in L^{r}\left ( \mathbb{R}^n  \right ) $,
		\begin{eqnarray}\label{e1.2}
			C\left ( n,\alpha ,p \right )\left \| f \right \|_{p}\left \| h \right \|_{r}\le n\omega^{\frac{n-\alpha }{n} }_{n}\left | S_{\alpha  }\left ( f,h \right )   \right |^{\frac{\alpha }{n} }\le\int_{\mathbb{R}^n}\int_{\mathbb{R}^n}\frac{f\left ( x \right )h\left ( y \right )  }{\left | x-y \right |^{n-\alpha }  }dxdy .
		\end{eqnarray}
		If $p=r=\frac{2n}{n+\alpha } $, there is equality in the first inequality if and only if  $h$ coincides with a translation of $a_{0} f$ where $a_{0}\in \mathbb{R}$ and
		\begin{eqnarray}
			f\left ( x \right )=a\left ( b ^{2}+\left | \phi \left ( x-x_{0} \right )^{2}    \right |   \right )^{-\frac{n+\alpha }{2} }\notag
		\end{eqnarray}
		for some $a\in \mathbb{R}, 0\ne b \in \mathbb{R}$,  $\phi \in GL\left ( n \right )$   and $x_{0}\in \mathbb{R }^{n}  $. There is equality in the second inequality if  $f$, $h$ are radially symmetric.
	\end{thm}
	The generalized affine Hardy–Littlewood–Sobolev inequality (\ref{e1.2}) is stronger than sharp  Hardy–Littlewood–Sobolev inequality (\ref{e1.6}). When $p=r=\frac{2n}{n+\alpha } $ and $f=h$, we can get the sharp affine Hardy-Littlewood-Sobolev inequalities (\ref{1bbb}). 
	
	In Section \ref{6}, for measurable sets $F\subseteq E\subseteq \mathbb{R}^n$, we will consider $S_{\alpha }\left ( 1_{E},1_{F}   \right )  $, where $1_{E}$ and $1_{F}$ denote the indicator functions of $E$ and $F$, respectively. For $n$-dimensional convex bodies $F\subseteq E\subseteq \mathbb{R}^n$, we demonstrate that $S_{\alpha }\left ( 1_{E},1_{F}   \right ) $ is proportional to the  generalized radial $\alpha$-mean body of $E$ and $F$, a significant concept introduced by R. Gardner and G. Zhang \cite[p.524]{GZ98}.
	In \cite{JM22-}, sharp isoperinetric inequalities were recently proved for radial $\alpha$-mean bodies, which is special case of generalized radial $\alpha$-mean body for $E=F$. As shown by R. Gardner and G. Zhang \cite{GZ98}, sharp reverse inequalities hold. They extend Zhang's reverse Petty projection inequality \cite{GZ91}, with equality being achieved exactly for $n$-dimensional simplices. Moreover, we introduce the
	conceptions of radial mean bodies in term of functions.
	
	For $0< \alpha <n$, J. Haddad and M. Ludwig \cite[Theorem 2]{JM22} etablished the sharp reverse inequalities for sharp affine HLS inequalities (\ref{1b}) as follows:
	\begin{eqnarray}\label{1.8}
		\frac{\Gamma \left ( n+1 \right )^{\frac{\alpha }{n} }  }{\Gamma \left ( \alpha  \right ) }\left | S_{\alpha } f   \right |^{\frac{\alpha }{n} }\ge   {\left \| f \right \|}_{2} ^{2-\frac{2\alpha }{n} }{\left \| f \right \|}_{1} ^{\frac{2\alpha }{n} }
		\ge {\left \| f \right \|}_{\frac{2n }{n+\alpha}}^{2}
	\end{eqnarray}
	for log-concave $f \in L^{2}\left ( \mathbb{R}^n  \right )$. For $\alpha > n$,
	J. Haddad and M. Ludwig \cite[Theorem 15]{JM22} etablished the sharp reverse inequalities for sharp affine HLS inequalities (\ref{1bbb}) as follows:
	\begin{eqnarray}\label{1.11}
		\frac{\Gamma \left ( n+1 \right )^{\frac{\alpha }{n} }  }{\Gamma \left ( \alpha  \right ) }\left | S_{\alpha } f   \right |^{\frac{\alpha }{n} }\le   {\left \| f \right \|}_{2} ^{2-\frac{2\alpha }{n} }{\left \| f \right \|}_{1} ^{\frac{2\alpha }{n} }
		\le {\left \| f \right \|}_{\frac{2n }{n+\alpha}}^{2}
	\end{eqnarray}
	for log-concave $f \in L^{2}\left ( \mathbb{R}^n  \right )$.
	The first inequality in (\ref{1.8}) and (\ref{1.11}) achieves equality if and only if $f\left ( x \right )=ae^{-\left \| x-x_{0}  \right \|_{\bigtriangleup }  }$ for $x\in \mathbb{R}^n  $, where $\alpha \ge 0,x_{0}\in \mathbb{R}^n $ and $\bigtriangleup $ an $n$-dimensional simplex having a vertex at the origin. 
	
	In Section \ref{7}, we prove reverse inequalities for the generalized affine HLS inequalities (\ref{e1.3}) and (\ref{e1.2}) within the functional framework as well.
	
	\begin{thm}\label{1.3}
		Let $f ,h\in L^{2}\left ( \mathbb{R}^n  \right )$ be even,  log-concave functions, then the following inequalities hold:
		for $0< \alpha <  n$,
		\begin{eqnarray}\label{1.333}
			\frac{\Gamma \left ( n+1 \right )^{\frac{\alpha }{n} }  }{\Gamma \left ( \alpha  \right ) }\left | S_{\alpha }\left ( f,h \right )   \right |^{\frac{\alpha }{n} }&\ge&  \left ( \left \| f \right \|_{1}\left \| h \right \|_{1}    \right )^{\frac{\alpha }{n} }\left ( \int_{\mathbb{R}^n }f\left ( x \right )h\left ( x \right )dx    \right )^{1-\frac{\alpha }{n} }\nonumber\\
			&\ge& {\left \| {f}^{\frac{n+\alpha }{2n}} {h}^{\frac{n-\alpha }{2n}}\right \|}_{\frac{2n }{n+\alpha}}{\left \| {h}^{\frac{n+\alpha }{2n}} {f}^{\frac{n-\alpha }{2n}}\right \|}_{\frac{2n }{n+\alpha}},
		\end{eqnarray}
		and for $\alpha>  n$,
		\begin{eqnarray}\label{1.12}
			\frac{\Gamma \left ( n+1 \right )^{\frac{\alpha }{n} }  }{\Gamma \left ( \alpha  \right ) }\left | S_{\alpha }\left ( f,h \right )   \right |^{\frac{\alpha }{n} }&\le& \left ( \left \| f \right \|_{1}\left \| h \right \|_{1}    \right )^{\frac{\alpha }{n} }\left ( \int_{\mathbb{R}^n }f\left ( x \right )h\left ( x \right )dx    \right )^{1-\frac{\alpha }{n} }\nonumber\\
			&\le&  {\left \| {f}^{\frac{n+\alpha }{2n}} {h}^{\frac{n-\alpha }{2n}}\right \|}_{\frac{2n }{n+\alpha}}{\left \| {h}^{\frac{n+\alpha }{2n}} {f}^{\frac{n-\alpha }{2n}}\right \|}_{\frac{2n }{n+\alpha}}. 
		\end{eqnarray}
		The first inequalities achieve equalities in (\ref{1.333}) and (\ref{1.12}) if and only if $f\left ( x \right )=h\left ( x \right )=ae^{-\left \| x-x_{0}  \right \|_{\bigtriangleup }  }$ for $x\in \mathbb{R}^n  $, where $a \ge 0$, $x_{0}\in \mathbb{R}^n $ and $\bigtriangleup $ an n-dimensional simplex having a vertex at the origin.
		\end{thm}
		When $f=h$ in inequalities (\ref{1.333}) and (\ref{1.12}), we can get inequalities (\ref{1.8}) and (\ref{1.11}). In Section \ref{s7.3}, we also get similar results as Theorem \ref{1.3} for $s$-concave functions.

\section{Preliminaries}\label{s2}
	In this section we introduce some notation and collect some basic facts from convex geometric analysis. On the theory of convex bodies, we can see the books by R. Gardner  \cite{R. Gardner06}, P. M. Gruber \cite{PG07} and R. Schneider \cite{Schneider14}.
	
	\subsection{Star bodies and dual mixed volumes}
	
	A set $ C \subseteq \mathbb{R}^{n} $ is {\it convex} if $ \left( 1-\lambda \right)x + \lambda y \in C $ holds for all $ x$, $y \in C $ and $ \lambda \in \left[ 0,1 \right] $. A {\it convex body} is a compact convex subset of $\mathbb{R}^{n}$ with a non-empty interior.
	
	A closed set $K \subseteq \mathbb{R}^{n}$ is said to be {\it star-shaped} (with respect to the origin) if the interval $\left[ 0,x \right] \subseteq K$ for every $x \in K$.
	The {\it radial function} $\rho_K:\mathbb{R}^n \setminus \{0\} \to [0,\infty ]$ associated with a star-shaped set $K$ is defined by
	\begin{eqnarray}
		{\rho}_{K}(x)=\sup\{ \lambda \geq 0:\lambda x\in K\}\nonumber
	\end{eqnarray}
	and the {\it gauge function} $\left \| \cdot  \right \|_{K}:\mathbb{R}^{n}\to \left [ 0,\infty  \right ] $ is given by 
	\begin{eqnarray}
		{||\cdot ||}_{K}=\inf\{ \lambda >0:x\in\lambda K\}.\nonumber
	\end{eqnarray}
	Two star-shaped sets $K$ and $L$ are said to be dilates if their radial functions satisfy $ {\rho}_{K} = c{\rho}_{L}$ almost everywhere on $\mathbb{S}^{n-1}$ for some $c \geq 0$.
	
	We call $K$ a {\it star body} if its radial function is strictly positive and continuous in $\mathbb{R}^n \setminus \{0\}$.
	We have ${||\cdot ||}_{B^{n}} = |\cdot |$, where $B^{n}$ is $n$-dimensional unit ball. Let $|K|$ denote $n$-dimensional Lebesgue measure of $K$. For a star-shaped set $K \subseteq \mathbb{R}^{n}$ with a measurable radial function, the $n$-dimensional Lebesgue measure of $K$ satisfies 
	\begin{equation}\label{e2.1}
	|K|=\frac{1}{n}\int_{\mathbb{S}^{n-1}}{\rho}_{K}(\xi)^{n}d\xi.
\end{equation}
	Let $\alpha \in (0, \infty)$ and $\alpha \neq n$. For star-shaped sets $K, L \subseteq \mathbb{R}^n$ endowed with measurable radial functions, the {\it dual mixed volume} is defined by
	\begin{eqnarray}\label{2.1a}
		\widetilde{V}_{\alpha}\left ( K,L \right )=\frac{1}{n}\int_{\mathbb{S}^{n-1} }\rho_{K }\left ( \xi  \right ) ^{n-\alpha }\rho_{L}\left ( \xi  \right )^{\alpha }d\xi.
	\end{eqnarray}
When $K=L$ in (\ref{2.1a}), we have
\begin{eqnarray}
\widetilde{V}_{\alpha}\left (K,K\right)=|K|.\nonumber
\end{eqnarray}
The concept of dual mixed volume for star bodies was introduced by E. Lutwak \cite{EL75}, who also established the dual mixed volume inequalities (\ref{2a}) and (\ref{2b}) as consequences of H\"older’s inequality. For $0 < \alpha < n$, the dual mixed volume inequality associated with star-shaped sets $K, L \subseteq \mathbb{R}^n$ of finite volume shows that
\begin{eqnarray}\label{2a}
	\widetilde{V}_{\alpha} (K,L)\leq |K|^{\frac{n-\alpha}{n}}|L|^{\frac{\alpha}{n}}.
\end{eqnarray}
For  $\alpha> n$, the dual mixed volume inequality associated with star-shaped sets $K, L \subseteq \mathbb{R}^n$ shows that
\begin{eqnarray}\label{2b}
\widetilde{V}_{\alpha} (K,L)\geq |K|^{\frac{n-\alpha}{n}}|L|^{\frac{\alpha}{n}}.
\end{eqnarray}
The equality condition in H\"older’s inequality implies that equalities in (\ref{2a}) and (\ref{2b}) for finite $\widetilde{V}_{\alpha}\left (K,L\right)$ are attained if and only if $K$ and $L$ are are dilates. For further details regarding dual mixed volumes, refer to \cite{R. Gardner06,Schneider14}.

\subsection{Log-concave and $s$-concave functions}
For $s>0$, a function $f:\mathbb{R}^n\to \left [ 0,\infty  \right )$ is called {\it $s$-concave} if
\begin{eqnarray}\label{2.11}
f\left ( \left ( 1-\lambda  \right )x+\lambda y  \right )^{s }\ge  \left ( 1-\lambda  \right )f\left ( x \right )^{s}+\lambda f\left ( y \right )^{s}  
\end{eqnarray}
for all $x,y\in \mathbb{R}^n $ and $0\le \lambda \le 1  $. Note that if $s> 0$, then $f^{s} $ is concave, and in particular, $1$-concave is just concave in the usual sense. If $s\rightarrow 0^+$, inequality (\ref{2.11}) becomes 
\begin{eqnarray}
f\left ( \left ( 1-\lambda  \right )x+\lambda y  \right )\ge f\left ( x \right )^{1-\lambda } f\left ( y \right )^{\lambda }  
\end{eqnarray}
which implies $f$ is {\it log-concave}.
For $f,h\in L^{1}\left ( \mathbb{R}^n  \right )$, the {\it convolution} of $f,h$ is defined by
\begin{eqnarray}
f\ast h\left ( x \right )=\int_{\mathbb{R}^n }f\left ( x-y \right )h\left ( y \right )dy  ,\;\;\; x\in \mathbb{R}^n.\notag		
\end{eqnarray}

The following result is a consequence of the Borell-Brascamp-Lieb inequality, see \cite[Theorem 11.3]{Gardner}.
\begin{lem}\label{2aa}
For $s\ge 0 $. If $ f,h\in L^{1}\left ( \mathbb{R}^n  \right ) $ are s-concave, then their convolution is $s/ (ns+2) $-concave.
\end{lem}

\subsection{Symmetrization}
For a subset $E\subseteq \mathbb{R}^n$, the {\it indicator function} $1_{E}$ is characterized by $1_{E}\left ( x \right )=1$ when $x\in E$ and $1_{E}\left ( x \right )=0$ for all $x\notin E$. Let $E\subseteq \mathbb{R}^n$ be a Borel set with finite measure. The {\it Schwarz symmetral} of $E$, denoted by $E^{\star}$, is defined as a closed Euclidean ball centered at the origin, having the same volume as $E$.
Let $f:\mathbb{R}^n \to \mathbb{R}$ be a non-negative measurable function such that for any $t > 0$, the super-level set $\left \{f \ge t \right \}$ has finite measure. The {\it layer cake formula} asserts that 
\begin{eqnarray}\label{2c}
	f\left ( x \right )=\int_{0}^{\infty }1_{\left \{ f\ge t \right \} }\left ( x \right )dt
\end{eqnarray}
holds for almost every $x \in \mathbb{R}^n$. The {\it Schwarz symmetrization} $f^{\star}$ of $f$ is defined as 
\begin{eqnarray}
	f^{\star }\left ( x \right )=\int_{0}^{\infty }1_{\left \{ f\ge t \right \}^{\star }  }(x)dt   \notag
\end{eqnarray}
for $x \in \mathbb{R}^n$. Consequently, $f^{\star}$ is characterized almost everywhere by being radially symmetric and possessing superlevel sets whose measures coincide with those of $f$. Note that $f^{\star }$ is also called to the {\it symmetric decreasing rearrangement} of $f$. If for all $x$, $y$ with $|x|<|y|$, the inequality $f^{\star }(x)>f^{\star }(y)$ holds, we say that $f^{\star }$ is {\it strictly symmetric decreasing}.

The proofs of our results rely on the Riesz rearrangement inequality (see, for example, \cite[Theorem 3.7]{ELL01}).
\begin{thm} \label{Rri} (Riesz’s rearrangement inequality). For $f,g,k:\mathbb{R}^n\to \mathbb{R}$ non-negative, measurable functions with superlevel sets of finite measure,
\begin{eqnarray*}
\int_{\mathbb{R}^n}\int_{\mathbb{R}^n}f(x)k(x-y)g(y)dxdy\leq \int_{\mathbb{R}^n}\int_{\mathbb{R}^n}f^{\star }(x)k^{\star}(x-y)g^{\star}(y)dxdy.
\end{eqnarray*}
\end{thm}
We will make use of the characterization of equality cases in the Riesz rearrangement inequality which is from A. Burchard \cite{Burchard96}.
\begin{thm} \label{2.1Burchard}(Burchard). Let $A$, $B$ and $C$ be sets of finite positive measure in $\mathbb{R}^n$ and denote by $\alpha$, $\beta$ and $\gamma$ the radii of their Schwarz symmetrals $A^{\star}$, $B^{\star}$ and $C^{\star}$. For $|\alpha -\beta |<\gamma <\alpha +\beta$, there is equality in
\begin{eqnarray*}
\int_{\mathbb{R}^n}\int_{\mathbb{R}^n}1_{A}(y)1_{B}(x-y)1_{C}(x)dxdy\leq \int_{\mathbb{R}^n}\int_{\mathbb{R}^n}1_{A^{\star}}(y)1_{B^{\star}}(x-y)1_{C^{\star}}(x)dxdy
\end{eqnarray*}
if and only if, up to sets of measure zero,
\begin{eqnarray*}
A=a+\alpha D,\; B=b+\beta D,\; C=c+\gamma D,
\end{eqnarray*}
where $D$ is a centered ellipsoid, and $a$, $b$ and $c = a + b$ are vectors in $\mathbb{R}^n$.
\end{thm}
\begin{thm} \label{2.1} 
(Burchard). Let $f,g,k:\mathbb{R}^n\to \mathbb{R}$ non-negative, non-zero, measurable functions with superlevel sets of finite measure such that
\begin{eqnarray*}
\int_{\mathbb{R}^n}\int_{\mathbb{R}^n}f(x)k(x-y)g(y)dxdy<\infty. 
\end{eqnarray*}
If at least two of the Schwarz symmetrals $f^{\star }$, $g^{\star}$, $k^{\star}$ are strictly symmetric decreasing, then there is equality
\begin{eqnarray*}
\int_{\mathbb{R}^n}\int_{\mathbb{R}^n}f(x)k(x-y)g(y)dxdy\leq \int_{\mathbb{R}^n}\int_{\mathbb{R}^n}f^{\star }(x)k^{\star}(x-y)g^{\star}(y)dxdy
\end{eqnarray*}
if and only if there is a volume-preserving $\phi\in GL(n)$ and $a$, $b$, $c \in \mathbb{R}^n$ with $c = a+b$ such that
\begin{eqnarray*}
f(x)=f^{\star }(\phi^{-1}x-a), k(x)=k^{\star}(\phi^{-1}x-b), g(x)=g^{\star}(\phi^{-1}x-c) 
\end{eqnarray*}
for $x \in \mathbb{R}^n$.
\end{thm}

\section{The star-shaped set ${S}_{\alpha }(f,h)$}\label{s5}
Let $f, h: \mathbb{R}^n \to [0, \infty)$ be measurable functions, $K \subseteq \mathbb{R}^n$ be a star-shaped set with a measurable radial function and $\alpha > 0$. Anisotropic fractional Sobolev norms were initially introduced in \cite{M. Ludwig14,M. Ludwig2014} and subsequently utilized in \cite{JM22}. In this work, we define
\begin{eqnarray*}
	\int_{\mathbb{R}^n}\int_{\mathbb{R}^n}\frac{f(x)h(y)}{{\left \| x-y\right \|}_{K }^{n-\alpha}}dxdy,
\end{eqnarray*}
an anisotropic version of the functional appearing in (\ref{1a}). By Fubini’s theorem, polar coordinates, and (\ref{1c}), we get that 
\begin{eqnarray*}
\int_{\mathbb{R}^n}\int_{\mathbb{R}^n}\frac{f(x)h(y)}{{\left \| x-y\right\|}_{K}^{n-\alpha}}dxdy&=&\int_{\mathbb{R}^n}\int_{\mathbb{R}^n}\frac{f(y+z)h(y)}{{\left\|z\right\|}_{K }^{n-\alpha}}dzdy\nonumber\\
&=&\int_{\mathbb{R}^n}\int_{\mathbb{S}^{n-1}}\int_{0}^{\infty }\frac{f(y+t\xi )h(y){t}^{n-1}}{{\left \| t\xi\right \|}_{K }^{n-\alpha}}dtd\xi dy\nonumber\\
&=&\int_{\mathbb{S}^{n-1}}{{\rho}_{K} (\xi)}^{n-\alpha}\int_{0}^{\infty }{t }^{\alpha-1}\int_{\mathbb{R}^n}f(y+t\xi )h(y)dydtd\xi\nonumber\\
&=&\int_{\mathbb{S}^{n-1}}{{\rho}_{K} (\xi)}^{n-\alpha}{{\rho}_{{S}_{\alpha }(f,h)} (\xi)}^{\alpha}d\xi.
\end{eqnarray*}
Then, for star-shaped $K\subseteq \mathbb{R}^n$ with measurable radial function and measurable functions $f,h:\mathbb{R}^n\to [0,\infty )$, we have
\begin{eqnarray}\label{3a}
\int_{\mathbb{R}^n}\int_{\mathbb{R}^n}\frac{f(x)h(y)}{{\left \| x-y\right \|}_{K }^{n-\alpha}}dxdy=n\widetilde{V}_{\alpha} (K,{S}_{\alpha }(f,h)).
\end{eqnarray}
For measurable $f,h:\mathbb{R}^n\to [0,\infty )$, by (\ref{e2.1}), (\ref{1c}), polar coordinates and Fubini’s theorem, we have
\begin{eqnarray}\label{3b}
|{S}_{n}(f,h)|&=&\frac{1}{n}\int_{\mathbb{S}^{n-1}}{{\rho}_{{S}_{n}(f,h)} (\xi)}^{n}d\xi\nonumber\\&=&\frac{1}{n}\int_{\mathbb{S}^{n-1}}\int_{0}^{\infty }{t}^{n-1}\int_{\mathbb{R}^n}f(x+t\xi )h(x)dxdtd\xi\nonumber\\
&=&\frac{1}{n}\int_{\mathbb{R}^n}\int_{\mathbb{R}^n}f(x+y )h(x)dxdy\nonumber\\
&=&\frac{1}{n}{||f||}_{1}{||h||}_{1}.
\end{eqnarray}
\begin{lem}\cite[Corollary 4.2]{GZ98}\label{3.11}
	Let $g$ be a nonnegative, integrable, log-concave function on ${\mathbb{R}}^{n}$ and  $\alpha > 0$, then the function $\rho$ defined by
	\begin{equation*}
		\rho(\xi)={\left(\int_{0}^{\infty }g(t\xi){t }^{\alpha-1}dt \right)}^{1/\alpha },
	\end{equation*}
	for $\xi \in {\mathbb{S}}^{n-1}$, is the extended radial function of a convex body in ${\mathbb{R}}^{n}$.
\end{lem}
We note the following property of ${S}_{\alpha }(f,h)$ for log-concave functions $f$ and $h$.
\begin{pro}
If $f,h:\mathbb{R}^n\to [0,\infty )$ are log-concave and in ${L}^{1} (\mathbb{R}^n)$, then ${S}_{\alpha }(f,h)$ is a convex body for every $\alpha >0$.
\end{pro}
\begin{proof}
Let $F(x,y)=f(x+y)h(x)$. Then $F(x,y)$ is log-concave on $(x,y)\in\mathbb{R}^{2n}$. In fact, since $f,h$ are log-concave and in ${L}^{1} (\mathbb{R}^n)$, we can see that
\begin{eqnarray}\label{3c}
F(x,y)
&=&f((1-\lambda){x}_{0}+\lambda{x}_{1}+(1-\lambda){y}_{0}+\lambda{y}_{1})h((1-\lambda){x}_{0}+\lambda{x}_{1})\nonumber\\
&\geq& {f({x}_{0}+{y}_{0})}^{1-\lambda }{f({x}_{1}+{y}_{1})}^{\lambda }{h({x}_{0})}^{1-\lambda }{h({x}_{1})}^{\lambda }\nonumber\\
&=&(f({x}_{0}+{y}_{0})h({x}_{0}))^{1-\lambda }(f({x}_{1}+{y}_{1})h({x}_{1}))^{\lambda }\nonumber\\
&=&F({x}_{0},{y}_{0})^{1-\lambda }F({x}_{1},{y}_{1})^{\lambda }
\end{eqnarray} 
where $x, y, {x}_{0}, {x}_{1}, {y}_{0}, {y}_{1}\in \mathbb{R}^n$, $x=(1-\lambda){x}_{0}+\lambda{x}_{1},y=(1-\lambda){y}_{0}+\lambda{y}_{1},0 \leq \lambda \leq 1$.

Let ${z}_{0}(x)=F(x,{y}_{0})$, ${z}_{1}(x)=F(x,{y}_{1})$, ${z}_{2}(x)=F(x,{y}_{2})$, ${y}_{2}=(1-\lambda){y}_{0}+\lambda{y}_{1}$, then by (\ref{3c}),
\begin{equation*}
{z}_{2}((1-\lambda){x}_{0}+\lambda{x}_{1})\geq{z}_{0}({x}_{0})^{1-\lambda }{z}_{1}({x}_{1})^{\lambda }.
\end{equation*} 
By Pr\'{e}kopa-Leindler inequality (see \cite[Theorem 4.2]{Gardner}), we have
\begin{eqnarray*}
\int_{\mathbb{R}^n}{z}_{2}(x)dx\geq \left(\int_{\mathbb{R}^n}z_{0}(x)dx \right)^{1-\lambda }\left(\int_{\mathbb{R}^n}z_{1}(x)dx\right )^{\lambda },
\end{eqnarray*} 
and thus
\begin{eqnarray*}
\int_{\mathbb{R}^n}f(x+(1-\lambda){y}_{0}+\lambda{y}_{1})h(x)dx\geq \left(\int_{\mathbb{R}^n}f(x+{y}_{0})h(x)dx\right)^{1-\lambda }\left(\int_{\mathbb{R}^n}f(x+{y}_{1})h(x)dx\right)^{\lambda }.
\end{eqnarray*}
Let $g(y)=\int_{\mathbb{R}^n}f(x+y)h(x)dx$, 
we have $g$ is log-concave.
By Lemma \ref{3.11}, we obtain the result.
\end{proof}

\section{Generalized affine HLS inequalities for $0<\alpha<n$}
The following result follows directly from the Riesz rearrangement inequality Theorem \ref{Rri} with case from Theorem \ref{2.1}.
\begin{lem}\label{4.1} Let $q > 0$ and let $K \subseteq \mathbb{R}^n$ be a star-shaped set with a measurable radial function and $|K| > 0$. Consider non-zero, measurable functions $f, h: \mathbb{R}^n \to [0, \infty)$ satisfying
	\begin{eqnarray*}
		\int_{\mathbb{R}^n}\int_{\mathbb{R}^n}\frac{f(x)h(y)}{{\left \| x-y\right \|}_{K }^{q}}dxdy<\infty 
	\end{eqnarray*}
	and $f^{\star}, h^{\star}$ are strictly symmetric decreasing rearrangements of $f$ and $h$, respectively. Then, equality holds in the inequality
	\begin{eqnarray*}
		\int_{\mathbb{R}^n}\int_{\mathbb{R}^n}\frac{f(x)h(y)}{{\left \| x-y\right \|}_{K }^{q}}dxdy \leq  \int_{\mathbb{R}^n}\int_{\mathbb{R}^n}\frac{f^{\star }(x)h^{\star }(y)}{{\left \| x-y\right \|}_{K ^{\star }}^{q}}dxdy
	\end{eqnarray*}
	if and only if $K$ is an ellipsoid and there exists a $\phi\in GL(n)$ such that $f$ is a translation of $f^{\star} \circ \phi^{-1}$ and $h$ is a translation of $h^{\star} \circ \phi^{-1}$. 
\end{lem}
The proof of Theorem \ref{1.1} relies on the following lemmas.
\begin{lem}\label{4.2}  Let $0 < \alpha < n$ and let $K \subseteq \mathbb{R}^n$ be a star-shaped set with a measurable radial function. Suppose $f, h: \mathbb{R}^n \to [0, \infty)$ are non-zero measurable functions satisfying
	\begin{eqnarray*}
		\int_{\mathbb{R}^n }\int_{\mathbb{R}^n} \frac{f(x)h(y)}{{\left \| x-y \right \|}_{K}^{n-\alpha}} dxdy < \infty.
		\end{eqnarray*}
		Then, the following inequality holds:
		\begin{eqnarray*}
			\widetilde{V}_{\alpha} (K, {S}_{\alpha}(f, h)) \leq \widetilde{V}_{\alpha} (K^{\star}, {S}_{\alpha}(f^{\star}, h^{\star})).
		\end{eqnarray*}
		Moreover, for $|K| > 0$ and $f^{\star },h^{\star }$ strictly symmetric decreasing, equality is achieved if and only if $K$ is an ellipsoid and there exists a $\phi\in GL(n)$ such that $f$ is a translation of $f^{\star} \circ \phi^{-1}$ and $h$ is a translation of $h^{\star} \circ \phi^{-1}$. 
\end{lem}
\begin{proof}
By (\ref{3a}) and the Riesz rearrangement inequality, Theorem \ref{Rri}, we obtain
\begin{eqnarray*}
	\widetilde{V}_{\alpha} (K,{S}_{\alpha }(f,h))
	&=&\frac{1}{n}\int_{\mathbb{R}^n}\int_{\mathbb{R}^n}\frac{f(x)h(y)}{{\left \| x-y\right \|}_{K }^{n-\alpha}}dxdy\nonumber\\
	&\leq&  \frac{1}{n}\int_{\mathbb{R}^n}\int_{\mathbb{R}^n}\frac{f(x)^{\star }h(y)^{\star }}{{\left \| x-y\right \|}_{K ^{\star }}^{n-\alpha}}dxdy= \widetilde{V}_{\alpha} (K^{\star },{S}_{\alpha }(f^{\star },h^{\star })).
\end{eqnarray*}
According to Lemma \ref{4.1}, equality holds if and only if $K$ is an ellipsoid and there exists a volume-preserving $\phi\in GL(n)$ such that $f$ is a translation of $f^{\star} \circ \phi^{-1}$ and $h$ is a translation of $h^{\star} \circ \phi^{-1}$.
\end{proof}
\begin{lem}\label{l4.3} Let $n\ge 1$, $1<p,r<\infty $, and $0<\alpha <n$ satisfy
	\begin{eqnarray}
		\frac{1}{p}+\frac{1}{r}-\frac{\alpha }{n}=1.\nonumber   
	\end{eqnarray}
	For all non-negative functions $f\in L^{p}\left ( \mathbb{R}^n  \right ) $ and $h\in L^{r}\left ( \mathbb{R}^n  \right ) $,
	\begin{eqnarray*}
		\left | S_{\alpha }\left ( f,h \right )   \right |\le  \left | S_{\alpha }\left ( f^{\star } ,h^{\star }  \right )   \right |.
	\end{eqnarray*}
	For $f^{\star },h^{\star }$ strictly symmetric decreasing with $| {S}_{\alpha }(f^{\star },h^{\star })| < \infty $, equality holds if and only if there exists a $\phi\in GL(n)$ such that $f$ and $h$ are translates  of $f^{\star }\circ \phi ^{-1}$ and $h^{\star }\circ \phi ^{-1}$, respectively.
\end{lem}
\begin{proof}
	First, suppose that $| {S}_{\alpha }(f,h)| < \infty $. Let $K = {S}_{\alpha }(f,h)$ in Lemma \ref{4.2} and by the dual mixed volume inequality (\ref{2a}) for $0 < \alpha < n$, we have
	\begin{eqnarray*}
		| {S}_{\alpha }(f,h)|&=&\widetilde{V}_{\alpha}\left ({S}_{\alpha }(f,h),{S}_{\alpha }(f,h)\right)\nonumber\\
		&\leq& \widetilde{V}_{\alpha}({({S}_{\alpha }(f,h))}^{\star },{S}_{\alpha }(f^{\star },h^{\star }))\nonumber\\
		&=& |({S}_{\alpha }(f,h))^{\star }|^{\frac{n-\alpha }{n}}|{S}_{\alpha }(f^{\star },h^{\star })|^{\frac{\alpha }{n}}\nonumber\\
		&=&|{S}_{\alpha }(f,h)|^{\frac{n-\alpha }{n}}|{S}_{\alpha }(f^{\star },h^{\star })|^{\frac{\alpha }{n}}.
	\end{eqnarray*} 
	Thus $\left | S_{\alpha }\left ( f,h \right )   \right |\le  \left | S_{\alpha }\left ( f^{\star } ,h^{\star }  \right )   \right |$. The case of equality is determined by Lemma \ref{4.2}.
	
	Second, suppose that $| {S}_{\alpha }(f,h)| =\infty $. For $k\ge 1
	$, define
	\begin{eqnarray*}
		{f}_{k}(x)=f(x){1}_{kB^{n}}(x),\;\;{h}_{k}(x)=h(x){1}_{kB^{n}}(x).
	\end{eqnarray*} 
Remark that ${f}_{k}, {h}_{k}$ are non-decreasing with respect to $k$ and converge pointwise to $f$ and $h$, respectively.
	Applying the monotone convergence theorem, we can see that
	\begin{eqnarray*}
		\lim_{k\to \infty }\int_{0}^{\infty }{t }^{\alpha-1}\int_{\mathbb{R}^n}f_{k}(x+t\xi )h_{k}(x)dxdt=\int_{0}^{\infty }{t }^{\alpha-1}\int_{\mathbb{R}^n}f(x+t\xi )h(x)dxdt
	\end{eqnarray*} 
	and the convergence is monotone. Applying the monotone convergence theorem again, we have
	\begin{eqnarray*}
		\lim_{k\to \infty }\int_{\mathbb{S}^{n-1}}\left(\int_{0}^{\infty }{t }^{\alpha-1}\int_{\mathbb{R}^n}f_{k}(x+t\xi )h_{k}(x)dxdt\right)^{\frac{n}{\alpha }}d\xi \\	=\int_{\mathbb{S}^{n-1}}\left(\int_{0}^{\infty }{t }^{\alpha-1}\int_{\mathbb{R}^n}f(x+t\xi )h(x)dxdt\right)^{\frac{n}{\alpha }}d\xi .
	\end{eqnarray*} 
	Therefore,
	\begin{eqnarray}\label{limk}
		\lim_{k\to \infty }| {S}_{\alpha }(f_{k},h_{k})|=| {S}_{\alpha }(f,h)|=\infty.
	\end{eqnarray} 
	Given that $f\in {L}^{p} (\mathbb{R}^n)$ and $h\in {L}^{r} (\mathbb{R}^n)$, the functions $f_{(k)}$ and $h_{(k)}$ possess compact support, and $| {S}_{\alpha }(f_{k},h_{k})| < \infty $ for all $k\ge 1$. It is easy to show that $(f_{k})^{\star } \le f^{\star }$ and $(h_{k})^{\star } \le h^{\star }$ almost everywhere. Therefore the first part of the lemma yields that
	\begin{eqnarray}\label{limkk}
		| {S}_{\alpha }(f_{k},h_{k})| \leq  | {S}_{\alpha }((f_{k})^{\star },(h_{k})^{\star })|\leq  | {S}_{\alpha }(f^{\star },h^{\star })|
	\end{eqnarray} 
	for $k\ge 1$. From (\ref{limk}) and (\ref{limkk}), it follows that $| {S}_{\alpha }(f^{\star },h^{\star })|=\infty$, as desired.
\end{proof}
\noindent\textbf{Proof of Theorem \ref{1.1}}.
Since ${||f||}_{p}={||f^{\star }||}_{p}$ and ${||h||}_{r}={||h^{\star }||}_{r}$, it follows from the classical HLS inequality (\ref{1.1}), (\ref{3a}), (\ref{2a}), and Lemma \ref{l4.3} that
\begin{eqnarray}\label{e4.2}
	C(n, \alpha, p){||f||}_{p}{||h||}_{r}&\geq&\int_{\mathbb{R}^n}\int_{\mathbb{R}^n}\frac{f^{\star }(x)h^{\star }(y)}{{\left| x-y\right |}^{n-\alpha}}dxdy\nonumber\\
	&=& n\widetilde{V}_{\alpha}\left ({B}^{n},{S}_{\alpha }(f^{\star },h^{\star })\right)\nonumber\\
	&=&n{\omega}_{n}^{\frac{n-\alpha}{n}}{|{S}_{n}(f^{\star },h^{\star })|}^{\frac{\alpha}{n}} \nonumber\\
	&\geq&n{\omega}_{n}^{\frac{n-\alpha}{n}}{|{S}_{\alpha}(f,h)|}^{\frac{\alpha}{n}}.
\end{eqnarray}
If $p=r=2n/n+\alpha $, there is equality in the first inequality (\ref{e1.3}) if and only if equalities in above ineqalities (\ref{e4.2}) hold throughout. 
When equality in the second inequality of (\ref{e4.2}) holds, by  the equality case of Lemma \ref{l4.3}, we have  that $h$ coincides with a translation of $a_{0} f$ and $f\left ( x \right )=a\left ( b ^{2}+\left | \phi \left ( x-x_{0} \right )^{2}    \right |   \right )^{-\frac{n+\alpha }{2} }$ for some $a_{0}$, $a\in \mathbb{R}$, $0\neq b \in\mathbb{R}$, $\phi\in GL(n)$ and $x_{0}\in \mathbb{R}^n$. Then $f^{\star }$ and $h^{\star }$ in terms of the above $f$ and $h$ can satisfy the equality case of the first inquelity in (\ref{e4.2}).
 
For the second inequality, let $K={B}^{n}$ in (\ref{3a}) and by the dual mixed volume inequality (\ref{2a}), we have
\begin{eqnarray*}
	\int_{\mathbb{R}^n}\int_{\mathbb{R}^n}\frac{f(x)h(y)}{{| x-y|}^{n-\alpha}}dxdy=n\widetilde{V}_{\alpha}({B}^{n},{S}_{\alpha }(f,h)) \leq n{\omega}_{n}^{\frac{n-\alpha}{n}}{|{S}_{\alpha}(f,h)|}^{\frac{\alpha}{n}}.
\end{eqnarray*} 
Equality in the second inequality holds if $f$ and $h$ are radially symmetric.
\qed

\section{Generalized affine HLS inequalities for $\alpha > n$}\label{5}
 We will prove generalized affine HLS inequalities for $\alpha > n$ which strengthen and imply (\ref{e1.6}). The following result follows from the Riesz rearrangement inequality and Theorem \ref{2.1Burchard}. Observe that the function $\left\|\cdot\right\|_{K}^{\alpha-n}$ in (\ref{5bbb}) possesses superlevel sets with infinite measure.
\begin{lem}\label{5a}
Let $\alpha > n$ and $ K\subseteq \mathbb{R}^n $ be a star-shaped set with $0< \left | K \right | < \infty $. For non-negative, non-zero functions $f\in L^{p}\left ( \mathbb{R}^n  \right )$, $h\in L^{r}\left ( \mathbb{R}^n  \right )$ and $\frac{1}{p}+\frac{1}{r}-\frac{\alpha }{n}=1$ satisfying
\begin{equation}
	\int_{\mathbb{R}^n}\int_{\mathbb{R}^n}\frac{f\left ( x \right )h\left ( y \right )  }{\left \| x-y \right \|^{n-\alpha } _{K} }dxdy  < \infty ,\label{5bbb}
\end{equation}
equality in 
\begin{eqnarray}
	\int_{\mathbb{R}^n}\int_{\mathbb{R}^n}\frac{f\left ( x \right )h\left ( y \right )  }{\left \| x-y \right \|^{n-\alpha } _{K}  }dxdy\ge\int_{\mathbb{R}^n}\int_{\mathbb{R}^n}\frac{f^{\star } \left ( x \right )h^{\star } \left ( y \right )  }{\left \| x-y \right \|^{n-\alpha } _{K^{\star } }  }dxdy\label{5cc}
\end{eqnarray}
holds if and only if $K$ is an ellipsoid, and there exists $\phi \in GL(n)$ such that $f$ coincides with a translation of $ f^{\star }\circ \phi^{-1} $and $h$  coincides with a translation of $h^{\star }\circ \phi^{-1}$.  
\end{lem}
\begin{proof}
We express the norm as 
\begin{eqnarray}
	\left \| z \right \|^{\alpha-n }_{K}=\int_{0}^{\infty }k_{s}\left ( z \right )  ds \notag  
\end{eqnarray}
where $k_{s}\left ( z \right )=1_{s^{1/(\alpha-n)  }\left ( \mathbb{R}^n\setminus K  \right )  }\left ( z \right )$. By the layer-cake formula (\ref{2c}) to the functions $f$ and $h$, we obtain 
\begin{eqnarray}
&&\int_{\mathbb{R}^n}\int_{\mathbb{R}^n}\frac{f\left ( x \right )h\left ( y \right )  }{\left \| x-y \right \|^{n-\alpha } _{K} }dxdy\notag\\
&=&\int_{0}^{\infty }\int_{0}^{\infty }\int_{0}^{\infty } \int_{\mathbb{R}^n }  \int_{\mathbb{R}^n }1_{\left \{ f\ge u \right \} }\left ( x \right )k_{s}\left ( x-y \right )1_{\left \{ h\ge t \right \} }(y)dxdydudsdt.\notag
\end{eqnarray}
By Theorem \ref{2.1Burchard}, it follows that
\begin{eqnarray}
&&\int_{\mathbb{R}^n }  \int_{\mathbb{R}^n }1_{\left \{ f\ge u \right \} }\left ( x \right )k_{s}\left ( x-y \right )1_{\left \{ h\ge t \right \} }(y)dxdy\notag\\
&=&\int_{\mathbb{R}^n }  \int_{\mathbb{R}^n }1_{\left \{ f\ge u \right \} }\left ( x \right )\left ( 1-1_{s^{1/(\alpha-n)  } K}\left ( x-y \right )   \right ) 1_{\left \{ h\ge t \right \} }(y)dxdy\notag\\
&=&\int_{\mathbb{R}^n }  \int_{\mathbb{R}^n }1_{\left \{ f\ge u \right \} }(x)1_{\left \{ h\ge t \right \} }(y)dxdy\notag\\
&-&\int_{\mathbb{R}^n }  \int_{\mathbb{R}^n }1_{\left \{ f\ge u \right \} }(x)1_{s^{1/(\alpha-n)  } K}\left ( x-y \right )1_{\left \{ h\ge t \right \} }(y)dxdy\notag\\
&\ge& \int_{\mathbb{R}^n }  \int_{\mathbb{R}^n }1_{\left \{ f\ge u \right \}^{\star }  }\left ( x \right )k^{\star } _{s}\left ( x-y \right )1_{\left \{ h\ge t \right \}^{\star }  }(y)dxdy\notag 
\end{eqnarray}
for $r,s,t> 0$. Given that $f\in L^{p}\left ( \mathbb{R}^n  \right )$ and $h\in L^{r}\left ( \mathbb{R}^n  \right )$, it follows that $\int_{\mathbb{R}^n }1_{\left \{ f\ge u \right \} }\left ( x \right )dx< \infty$ and $\int_{\mathbb{R}^n }1_{\left \{ h\ge t \right \} }\left ( y \right )dy< \infty$ for $r, u > 0$. If equality holds in (\ref{5cc}), thus there exists a null set $N\subseteq \left ( 0,\infty  \right )^{3}$ such that  
\begin{eqnarray}
	&&\int_{\mathbb{R}^n }  \int_{\mathbb{R}^n }1_{\left \{ f\ge u \right \} }\left ( x \right )1_{s^{1/(\alpha-n)  } K}\left ( x-y \right )  1_{\left \{ h\ge t \right \} }(y)dxdy\notag\\
	&=&\int_{\mathbb{R}^n }  \int_{\mathbb{R}^n }1_{\left \{ f\ge u \right \}^{\star }  }\left ( x \right )1_{s^{1/(\alpha-n)  } K^{\star }}\left ( x-y \right )  1_{\left \{ h\ge t \right \}^{\star }  }(y)dxdy\notag	 
\end{eqnarray}  
for all $\left ( u,s,t \right )\in \left ( 0,\infty  \right )^{3}\setminus N$. Moreover, for almost every $\left ( u,t \right ) \in \left ( 0,\infty  \right )^{2}$,  we can see that  $\left ( u,s,t \right )\in \left ( 0,\infty  \right )^{3}\setminus N $ for almost every $s> 0$. Let $\alpha$, $\beta$ and $\gamma$  are the radii of ${\left \{ f\ge u \right \}}^{\star}$, ${s^{1/(\alpha-n)}K}^{\star}$ and ${\left \{ h\ge t \right \}}^{\star}$, respectively. 
Fixed $u$ and $ s $, 
then there exists a closed interval $[{\bar{t}}_{u,s},{\tilde{t}}_{u,s}]$ with respect to $ u $ and $ s $ such that $ |\alpha-\beta| <\gamma<\alpha+\beta$ for any $t\in[{\bar{t}}_{u,s},{\tilde{t}}_{u,s}]$.  Thus, by Theorem \ref{2.1Burchard}, there exists an ellipsoid $D$ and vectors $a, b \in \mathbb{R}^n$ such that  
\begin{eqnarray}
	\left \{ f\ge u \right \}=a+\alpha D,\; s^{1/(\alpha-n)}K=b+\beta D,\; \left \{ h\ge t \right \}={c}_{t}+\gamma D\notag
\end{eqnarray}  
where ${c}_{t} = a + b$. Then for $t\in[{\bar{t}}_{u,s},{\tilde{t}}_{u,s}]$, ${c}_{t}$ is a constant vector. When $u$ and $s$ vary, the interval $\displaystyle\bigcup_{u,s>0} [{\bar{t}}_{u,s},{\tilde{t}}_{u,s}]$ covers $(0, +\infty)$, thus for any $t\in(0, +\infty)$, ${c}_{t}$ is a constant vector.  
Since $K = s^{1/(n-\alpha)}b + \left( \left| K \right| / \left| D \right| \right)^{\frac{1}{n}}D$, the ellipsoid $D$ is independent of $\left( u, s, t \right)$, and consequently, the vector $b$ is a constant vector. Thus, vector $a$ is also a constant vector. This completes the proof.
\end{proof}
\begin{lem}\label{5b}
Let $\alpha > n$ and $K\subseteq \mathbb{R}^n$ be star-sharped with measureable radial function. Suppose $f,h:\mathbb{R}^n \longrightarrow \left [ 0,\infty  \right ) $  are non-zero  measurable functions and
\begin{eqnarray}
\int_{\mathbb{R}^n}\int_{\mathbb{R}^n}\frac{f\left ( x \right )h\left ( y \right )  }{\left \| x-y \right \|^{n-\alpha } _{K} }dxdy  < \infty ,\notag
\end{eqnarray}
then the following inequality 
\begin{eqnarray}
\widetilde{V}_{\alpha } \left ( K,S_{\alpha }\left ( f,h \right )   \right )\ge\widetilde{V}_{\alpha } \left ( K^{\star } ,S_{\alpha }\left ( f^{\star } ,h^{\star }  \right )   \right )\notag
\end{eqnarray}
holds. Moreover, when $\left | K \right |> 0 $, equality occurs if and only if $K$ is an ellipsoid, and there exists  $\phi\in GL\left ( n \right ) $ such that $f$ coincides with a translation of $ f^{\star }\circ \phi^{-1} $and $h$  coincides with a translation of $h^{\star }\circ \phi^{-1}$.  
\end{lem}
\begin{proof}
Combining (\ref{3a})and Lemma \ref{5a}, we obtain the inequality
\begin{eqnarray*}
	\widetilde{V}_{\alpha} (K,{S}_{\alpha }(f,h))
	&=&\frac{1}{n}\int_{\mathbb{R}^n}\int_{\mathbb{R}^n}\frac{f(x)h(y)}{{\left \| x-y\right \|}_{K }^{n-\alpha}}dxdy\nonumber\\
	&\ge&  \frac{1}{n}\int_{\mathbb{R}^n}\int_{\mathbb{R}^n}\frac{f(x)^{\star }h(y)^{\star }}{{\left \| x-y\right \|}_{K ^{\star }}^{n-\alpha}}dxdy= n\widetilde{V}_{\alpha} (K^{\star },{S}_{\alpha }(f^{\star },h^{\star })).
\end{eqnarray*}
Furthermore, Lemma \ref{5a} implies that equality holds if and only if $K$ is a ellipsoid, and there exists  $\phi \in GL\left ( n \right ) $ such that $f$ coincides with a translation of $ f^{\star }\circ \phi^{-1} $and $h$  coincides with a translation of $h^{\star }\circ \phi^{-1}$. 
\end{proof}
\begin{lem}\label{5c}
Let $\alpha > n$ and consider measurable functions $f, h:\mathbb{R}^n \rightarrow \left [ 0,\infty  \right ) $. If $\left | S_{\alpha }\left ( f,h \right )   \right |<  \infty   $, then the following inequality holds:
\begin{eqnarray}
\left | S_{\alpha }\left ( f,h \right )   \right |\ge\left | S_{\alpha }\left ( f^{\star } ,h^{\star }  \right )   \right |.\notag
\end{eqnarray}
Equality occurs if and only if there exists  $\phi\in GL\left ( n \right ) $ such that $f$ coincides with a translation of $ f^{\star }\circ \phi^{-1} $and $h$  coincides with a translation of $h^{\star }\circ \phi^{-1}$.
\end{lem}
\begin{proof}
Applying Lemma \ref{5b} with $K=S_{\alpha }\left ( f,h \right ) $ and combining with the dual mixed volume inequality (\ref{2b}) for $\alpha > n$, we obtain
\begin{eqnarray}
\left | S_{\alpha }\left ( f,h \right )   \right | &=&\widetilde{V}_{\alpha }\left (S_{\alpha }\left ( f,h \right ),S_{\alpha }\left ( f,h \right ) \right )\notag\\
&\ge&  \widetilde{V}_{\alpha }\left (S_{\alpha }\left ( f,h \right )^{\star } ,S_{\alpha }\left ( f^{\star } ,h^{\star }  \right ) \right )\notag\\
&\ge&\left | \left ( S_{\alpha }\left ( f,h \right )    \right )^{\star }   \right |^{\frac{n-\alpha }{n}}\left | \left ( S_{\alpha }\left ( f^{\star } ,h^{\star }  \right )    \right )  \right |^{\frac{\alpha }{n}}\notag\\
&=&  \left | \left ( S_{\alpha }\left ( f,h \right )    \right )   \right |^{\frac{n-\alpha }{n}}\left | \left ( S_{\alpha }\left ( f^{\star }  ,h^{\star } \right )    \right )  \right |^{\frac{\alpha }{n}}.\notag 
\end{eqnarray}
Thus $\left | S_{\alpha }\left ( f,h \right )   \right |\ge\left | S_{\alpha }\left ( f^{\star } ,h^{\star }  \right )   \right |$.
The equality case is immediately obtained from Lemma \ref{5b}.
\end{proof}
We now establish the extended version of affine HLS inequalities for the case $\alpha > n$.

\noindent\textbf{Proof of Theorem \ref{1.2}}.
In order to obtain the first inequality, we suppose that $\left | S_{\alpha }\left ( f,h \right )   \right |< \infty $. By $\left \| f \right \|_{p}=\left \| f^{\star }  \right \|_{p}$, $\left \| h \right \|_{r}=\left \| h^{\star }  \right \|_{r} $, HLS inequality (\ref{e1.6}), (\ref{3a}), equality case of (\ref{2b}) and Lemma \ref{5c}, we have
\begin{eqnarray}\label{e5.2}
C\left ( n,\alpha ,p \right )\left \| f \right \|_{p}\left \| h \right \|_{r}&\le&\int_{\mathbb{R}^n}\int_{\mathbb{R}^n}\frac{f^{\star } \left ( x \right )h^{\star } \left ( y \right )  }{\left | x-y \right |^{n-\alpha }  }dxdy\notag\\
&=&n\widetilde{V}_{\alpha } \left ( B^{n},S_{\alpha  }\left ( f^{\star },h^{\star }   \right )    \right )\notag\\
&=& n\omega^{\frac{n-\alpha }{n} }_{n}\left | S_{\alpha  }\left ( f^{\star } ,h^{\star }  \right )   \right |^{\frac{\alpha }{n} }\notag\\
&\le&n\omega^{\frac{n-\alpha }{n} }_{n}\left | S_{\alpha  }\left ( f ,h  \right )   \right |^{\frac{\alpha }{n} }.
\end{eqnarray}
If $p=r=2n/n+\alpha $, there is equality in the first inequality (\ref{e1.2}) if and only if equalities in above ineqalities (\ref{e5.2}) hold throughout. 
When equality in the second inequality of (\ref{e5.2}) holds, by  the equality case of Lemma \ref{5c}, we have  that $h$ coincides with a translation of $a_{0} f$ and $f\left ( x \right )=a\left ( b ^{2}+\left | \phi \left ( x-x_{0} \right )^{2}    \right |   \right )^{-\frac{n+\alpha }{2} }$ for some $a_{0}$, $a\in \mathbb{R}$, $0\neq b \in\mathbb{R}$, $\phi\in GL(n)$ and $x_{0}\in \mathbb{R}^n$. Note that $f^{\star }$ and $h^{\star }$ in terms of the above $f$ and $h$ can satisfy the equality case of the first inquelity in (\ref{e5.2}).

For the second inequality, suppose 
\begin{eqnarray}
\int_{\mathbb{R}^n}\int_{\mathbb{R}^n}\frac{f\left ( x \right )h\left ( y \right )  }{\left | x-y \right |^{n-\alpha } }dxdy  < \infty .\notag
\end{eqnarray}
Taking $K=B^{n} $ in (\ref{3a}) and using the dual mixed volume inequality (\ref{2b}), we derive
\begin{eqnarray}
\int_{\mathbb{R}^n}\int_{\mathbb{R}^n}\frac{f\left ( x \right )h\left ( y \right )  }{\left | x-y \right |^{n-\alpha } }dxdy=n\widetilde{V}_{\alpha  }\left ( B^{n},S_{\alpha  }\left ( f,h   \right )    \right )\ge n\omega^{\frac{n-\alpha }{n} }_{n}\left | S_{\alpha  }\left ( f,h  \right )   \right |^{\frac{\alpha }{n} }.\notag
\end{eqnarray}
Equality occurs if and only if  $f$ and $h$ are radially symmetric functions.
\qed

\section{Radial Mean Bodies}\label{6}
In the context of convex geometry, consider two convex bodies $E$ and $F$ in $ \mathbb{R}^n$ such that $F\subseteq E$. For  $\alpha > -1$, R. Gardner and G. Zhang \cite{GZ98} introduced the generalized  radial $\alpha $-th mean body ${R}_{\alpha }(E,F)$, whose radial function for $\xi \in {\mathbb{S}}^{n-1} $ is defined by the equation:
\begin{eqnarray}\label{66}
{{\rho}_{{R}_{\alpha }(E,F)} (\xi)}^{\alpha }=\frac{1}{|F|}\int_{F}{{\rho}_{E-x} (\xi)}^{\alpha }dx.
\end{eqnarray} 

J. Haddad and M. Ludwig \cite{JM2025} defined the $s$-fractional $L^p$ polar projection body.
Let $0<s<1$ and $1<p<n / s$. For a measurable function $f: \mathbb{R}^n \rightarrow \mathbb{R}$, the $s$-fractional $L^p$ polar projection body $\Pi_p^{*, s} f$ as the star-shaped set is given by the radial function
$$
{\rho}_{\Pi_p^{*, s} f}^{-p s}(\xi)=\int_0^{\infty} t^{-p s-1} \int_{\mathbb{R}^n}|f(x+t \xi)-f(x)|^p \mathrm{~d} x \mathrm{~d} t
$$
for $\xi \in \mathbb{R}^n$. 

For measurable functions $f,h:\mathbb{R}^{n}\rightarrow \mathbb{R}$ and $-1<\alpha<0$, we define the generalized asymmetric  $\alpha$-fractional $L^2$ polar projection body of $f$ and $h$, written as ${\Pi }_{2,-}^{*,-\frac{\alpha }{2} }(f,h)$, by its radial function for $\xi \in {\mathbb{S}}^{n-1} $:
\begin{eqnarray}
{{\rho}_{ {\Pi }_{2,-}^{*,-\frac{\alpha }{2} }(f,h)} (\xi)}^{\alpha }=\int_{0}^{\infty }{t}^{\alpha-1}\int_{\mathbb{R}^n}{\left(f\left ( x+t\xi \right ) -h(x)\right)}^{2} _{-}dxdt.	
\end{eqnarray}
In the special case where $f=1_{E}$ and $h=1_{F} $,  we denote the generalized asymmetric fractional polar projection body of $1_{E}$ and $1_{F}$ by ${\Pi}_{2,-}^{*,-\frac{\alpha }{2} }(E,F)$. Note that its radial function satisfies:
\begin{eqnarray}\label{67}
{{\rho}_{ {\Pi }_{2,-}^{*,-\frac{\alpha }{2} }(E,F)} (\xi)}^{\alpha }=\int_{0}^{\infty }{t}^{\alpha-1}\int_{\mathbb{R}^n}{\left(1_{E}(x+t\xi)-1_{F}(x)\right)}_{-}dxdt.
\end{eqnarray} 
For two convex bodies $E$ and $F$ in $ \mathbb{R}^n$, let ${S}_{\alpha }(E,F)={S}_{\alpha }(1_{E},1_{F})$.
\begin{lem}\label{L6.1}
Let $E$, $F\subseteq \mathbb{R}^n$ be two convex bodies and $F\subseteq E $. For $\alpha>0$, we have
\begin{eqnarray}\label{6.4}
{S}_{\alpha }(E,F)=\left(\frac{|F|}{\alpha }\right)^{\frac{1}{\alpha}}{R}_{\alpha }(E,F).
\end{eqnarray}
\end{lem}
\begin{proof}
For $\alpha >0$, by the polar coordinates transformation, Fubini’s theorem, (\ref{66}) and (\ref{2.1a}), we can see that
\begin{eqnarray}\label{e6.1}
\int_{\mathbb{R}^n}\int_{\mathbb{R}^n}\frac{1_{E}(x)1_{F}(y)}{{\left \| x-y\right \|}_{K }^{n-\alpha}}dxdy&=&\int_{F}\int_{E}\frac{1}{{\left \| x-y\right \|}_{K }^{n-\alpha}}dxdy\nonumber\\
&=& \int_{F}\int_{E-y}\frac{1}{{{\left \| z\right \|}_{K }^{n-\alpha}}}dzdy\nonumber\\
&=&\int_{F}\int_{\mathbb{S}^{n-1}}\int_{0}^{{{\rho}_{E-y} (\xi)} }\frac{1}{{\left \|t\xi\right \|}_{K }^{n-\alpha}}{t}^{n-1}dtd\xi dy \nonumber\\
&=&\frac{1}{\alpha }\int_{F}\int_{\mathbb{S}^{n-1}}{{\rho}_{E-y} (\xi)}^{\alpha }{{\rho}_{K} (\xi)}^{n-\alpha }d\xi dy\nonumber\\
&=&\frac{|F|}{\alpha }\int_{\mathbb{S}^{n-1}}{{\rho}_{K} (\xi)}^{n-\alpha }\frac{1}{|F|}\int_{F}{{\rho}_{E-y} (\xi)}^{\alpha }dyd\xi\nonumber\\
&=&\frac{|F|}{\alpha }\int_{\mathbb{S}^{n-1}}{{\rho}_{K} (\xi)}^{n-\alpha }{{\rho}_{{R}_{\alpha }(E,F)} (\xi)}^{\alpha }d\xi\nonumber\\
&=&\frac{n|F|}{\alpha }\widetilde{V}_{\alpha} (K,{R}_{\alpha }(E,F)).
\end{eqnarray} 
Taking $f=1_{E}$ and $h=1_{F}$ in (\ref{3a}), we have
\begin{eqnarray}\label{e6.2}
\int_{\mathbb{R}^n}\int_{\mathbb{R}^n}\frac{1_{E}(x)1_{F}(y)}{{\left \| x-y\right \|}_{K }^{n-\alpha}}dxdy=n\widetilde{V}_{\alpha}(K,{S}_{\alpha }(E,F)).
\end{eqnarray}
By (\ref{e6.1}) and (\ref{e6.2}), we have
\begin{eqnarray}\label{e6a}
\widetilde{V}_{\alpha} (K,{S}_{\alpha }(E,F))=\frac{|F|}{\alpha }\widetilde{V}_{\alpha} (K,{R}_{\alpha }(E,F)).
\end{eqnarray} 
If $0<\alpha <n$, let $K={S}_{\alpha }(E,F)$ in (\ref{e6a}), and by dual mixed volume inequality (\ref{2a}), we obtain that
\begin{eqnarray*}
|{S}_{\alpha }(E,F)|=\frac{|F|}{\alpha }\widetilde{V}_{\alpha}({S}_{\alpha }(E,F),{R}_{\alpha }(E,F))\le \frac{|F|}{\alpha }|{S}_{\alpha }(E,F)|^{\frac{n-\alpha}{n}}|{R}_{\alpha }(E,F)|^{\frac{\alpha}{n}}, 
\end{eqnarray*} 
and thus
\begin{eqnarray}\label{6b}
|{S}_{\alpha }(E,F)|^{\frac{\alpha}{n}}\le \frac{|F|}{\alpha }|{R}_{\alpha }(E,F)|^{\frac{\alpha}{n}}.
\end{eqnarray} 
Let $K={R}_{\alpha }(E,F)$ in (\ref{e6a}), and by dual mixed volume inequality (\ref{2a}), we have
\begin{eqnarray*}
\frac{|F|}{\alpha }|{R}_{\alpha }(E,F)|=\widetilde{V}_{\alpha} ({R}_{\alpha }(E,F),{S}_{\alpha }(E,F))\le |{R}_{\alpha }(E,F)|^{\frac{n-\alpha}{n}}|{S}_{\alpha }(E,F)|^{\frac{\alpha}{n}},
\end{eqnarray*} 
and thus
\begin{eqnarray}\label{6c}
\frac{|F|}{\alpha }|{R}_{\alpha }(E,F)|^{\frac{\alpha}{n}} \le |{S}_{\alpha }(E,F)|^{\frac{\alpha}{n}}.
\end{eqnarray} 
By (\ref{6b}) and (\ref{6c}), we can see that
\begin{eqnarray}\label{e6.3}
|{S}_{\alpha }(E,F)|^{\frac{\alpha}{n}}=\frac{|F|}{\alpha }|{R}_{\alpha }(E,F)|^{\frac{\alpha}{n}}.
\end{eqnarray}
By the equality condition of dual mixed volume inequality (\ref{2a}) and (\ref{e6.3}), we have
\begin{eqnarray}\label{6d}
{S}_{\alpha }(E,F)=\left(\frac{|F|}{\alpha }\right)^{\frac{1}{\alpha}}{R}_{\alpha }(E,F)
\end{eqnarray} 
for $0<\alpha <n$. We can get the same result for $\alpha >n$ by (\ref{e6a}) and dual mixed volume inequality (\ref{2b}) with the same method.
\end{proof}
\begin{lem}\label{L6.2}
Let $E,F \subseteq \mathbb{R}^n$ be convex. For $-1<\alpha<0$, 
we have
\begin{eqnarray}\label{6.12}
{\Pi }_{2,-}^{*,-\frac{\alpha }{2} }(E,F)=\left(\frac{|F|}{|\alpha| }\right)^{\frac{1}{\alpha}}{R}_{\alpha }(E,F).
\end{eqnarray}
\end{lem}
\begin{proof}
For $-1<\alpha<0$, by the polar coorcoordinates transformation, Fubini’s theorem, (\ref{66}) and (\ref{2.1a}), we obtain
\begin{eqnarray}\label{ee6.1}
\int_{\mathbb{R}^n}\int_{\mathbb{R}^n}\frac{{(1_{E}(x)-1_{F}(y))}_{- }}{{\left \| x-y\right \|}_{K }^{n-\alpha}}dxdy&=&\int_{F}\int_{E^{c}}\frac{1}{{\left \| x-y\right \|}_{K }^{n-\alpha}}dxdy\nonumber\\
&=& \int_{F}\int_{E^{c}-y}\frac{1}{{{\left \| z\right \|}_{K }^{n-\alpha}}}dzdy\nonumber\\
&=&\int_{F}\int_{\mathbb{S}^{n-1}}\int_{{{\rho}_{E-y} (\xi)}}^{\infty }{r}^{\alpha-1}{{\rho}_{K} (\xi)}^{n-\alpha }drd\xi dy \nonumber\\
&=&\frac{|F|}{-\alpha }\int_{\mathbb{S}^{n-1}}{{\rho}_{K} (\xi)}^{n-\alpha }\frac{1}{|F|}\int_{F}{{\rho}_{E-y} (\xi)}^{\alpha }dyd\xi\nonumber\\
&=&\frac{|F|}{|\alpha |}n\widetilde{V}_{\alpha} (K,{R}_{\alpha }(E,F)).
\end{eqnarray} 
Moreover, by the polar coordinates transformation, Fubini’s theorem, (\ref{67}) and (\ref{2.1a}), we can see that
\begin{eqnarray}\label{ee6.2}
\int_{\mathbb{R}^n}\int_{\mathbb{R}^n}\frac{{(1_{E}(x)-1_{F}(y))}_{- }}{{\left \| x-y\right \|}_{K }^{n-\alpha}}dxdy&=&\int_{\mathbb{R}^n}\int_{\mathbb{R}^n}\frac{{(1_{E}(y+z)-1_{F}(y))}_{- }}{{\left \| z\right \|}_{K }^{n-\alpha}}dzdy\nonumber\\
&=&\int_{\mathbb{R}^n}\int_{\mathbb{S}^{n-1}}\int_{0}^{\infty }\frac{{(1_{E}(y+t\xi)-1_{F}(y))}_{- }}{{\left \|t\xi\right \|}_{K }^{n-\alpha}}{t}^{n-1}dtd\xi dy\nonumber\\
&=&\int_{\mathbb{S}^{n-1}}{{\rho}_{K} (\xi)}^{n-\alpha}\int_{0}^{\infty }{t}^{\alpha-1}\int_{\mathbb{R}^n}{(1_{E}(y+t\xi)-1_{F}(y))}_{-}dydtd\xi\nonumber\\
&=&\int_{\mathbb{S}^{n-1}}{{\rho}_{K} (\xi)}^{n-\alpha}{{\rho}_{{\Pi}_{2,-}^{*,-\frac{\alpha }{2} }} (\xi)}^{\alpha}d\xi\nonumber\\
&=&n\widetilde{V}_{\alpha}(K,{\Pi }_{2,-}^{*,-\frac{\alpha }{2} }(E,F)).
\end{eqnarray} 
By (\ref{ee6.1}) and (\ref{ee6.2}), we have
\begin{eqnarray}\label{6e}
\widetilde{V}_{\alpha}(K,{\Pi }_{2,-}^{*,-\frac{\alpha }{2} }(E,F))=\frac{|F|}{|\alpha |}\widetilde{V}_{\alpha} (K,{R}_{\alpha }(E,F)).
\end{eqnarray} 
Let $K={\Pi }_{2,-}^{*,-\frac{\alpha }{2} }(E,F)$ in (\ref{6e}), and by dual mixed volume inequality (\ref{2b}), we obtain that
\begin{eqnarray*}
|{\Pi }_{2,-}^{*,-\frac{\alpha }{2} }(E,F)|=\frac{|F|}{|\alpha| }\widetilde{V}_{\alpha} ({\Pi }_{2,-}^{*,-\frac{\alpha }{2} }(E,F),{R}_{\alpha }(E,F))\geq  \frac{|F|}{|\alpha |}|{\Pi }_{2,-}^{*,-\frac{\alpha }{2} }(E,F)|^{\frac{n-\alpha}{n}}|{R}_{\alpha }(E,F)|^{\frac{\alpha}{n}}, 
\end{eqnarray*} 
and thus
\begin{eqnarray}\label{6f}
|{\Pi }_{2,-}^{*,-\frac{\alpha }{2} }(E,F)|^{\frac{\alpha}{n}}\geq \frac{|F|}{|\alpha |}|{R}_{\alpha }(E,F)|^{\frac{\alpha}{n}}.
\end{eqnarray} 
Let $K={R}_{\alpha }(E,F) $ in (\ref{6e}), and by dual mixed volume inequality (\ref{2b}), we have
\begin{eqnarray*}
\frac{|F|}{|\alpha| }|{R}_{\alpha }(E,F)|=\widetilde{V}_{\alpha} ({R}_{\alpha }(E,F),{\Pi }_{2,-}^{*,-\frac{\alpha }{2} }(E,F))\geq |{R}_{\alpha }(E,F)|^{\frac{n-\alpha}{n}}|{\Pi }_{2,-}^{*,-\frac{\alpha }{2} }(E,F)|^{\frac{\alpha}{n}},
\end{eqnarray*} 
and thus
\begin{eqnarray}\label{6g}
\frac{|F|}{|\alpha| }|{R}_{\alpha }(E,F)|^{\frac{\alpha}{n}} \geq |{\Pi }_{2,-}^{*,-\frac{\alpha }{2} }(E,F)|^{\frac{\alpha}{n}}.
\end{eqnarray} 
By (\ref{6f}) and (\ref{6g}), we can see that
\begin{eqnarray}\label{eee6.3}
|{\Pi }_{2,-}^{*,-\frac{\alpha }{2} }(E,F)|^{\frac{\alpha}{n}}=\frac{|F|}{|\alpha| }|{R}_{\alpha }(E,F)|^{\frac{\alpha}{n}}.
\end{eqnarray}
By the equality condition of dual mixed volume inequality (\ref{2b}) and (\ref{eee6.3}), we have
\begin{eqnarray}\label{6h}
{\Pi }_{2,-}^{*,-\frac{\alpha }{2} }(E,F)=\left(\frac{|F|}{|\alpha| }\right)^{\frac{1}{\alpha}}{R}_{\alpha }(E,F)
\end{eqnarray} 
for $-1<\alpha <0$.
\end{proof}
By Lemma \ref{L6.1} and Lemma \ref{L6.2}, for non-zero, non-negative functions $f\in {L}^{2} (\mathbb{R}^n)$ and $h\in {L}^{2} (\mathbb{R}^n)$, we can define radial mean bodies ${R}_{\alpha }(f,h)$ of functions $f$ and $h$ as follows
\begin{eqnarray}\label{7a}
	{R}_{\alpha }(f,h)=\left(\frac{\alpha}{\int_{\mathbb{R}^n} fhdx}\right)^{\frac{1}{\alpha}}{S}_{\alpha }(f,h)
\end{eqnarray}
for $\alpha > 0$, while for $-1<\alpha<0$,
\begin{eqnarray}\label{7b}
	{R}_{\alpha }(f,h)=\left(\frac{|\alpha|}{{\int_{\mathbb{R}^n} }fhdx }\right)^{\frac{1}{\alpha}}{\Pi }_{2,-}^{*,-\frac{\alpha }{2} }(f,h).
\end{eqnarray}	
The definitions  (\ref{7a}) and  (\ref{7b}) remain consistent with (\ref{6.4}) and  (\ref{6.12}) when $f={1}_{E}$, $h={1}_{F}$ and convex bodies $F \subseteq E \subseteq \mathbb{R}^n$. Furthermore, from (\ref{7a}) and (\ref{3b}) we obtain
\begin{eqnarray}\label{7c}
	|{R}_{n}(f,h)|=\frac{{||f||}_{1}{||h||}_{1}}{{\int_{\mathbb{R}^n} }fhdx}
\end{eqnarray}	
for non-zero functions $f\in {L}^{2} (\mathbb{R}^n)$ and $h\in {L}^{2} (\mathbb{R}^n)$.
\section{Generalized reverse affine HLS inequalities}\label{7}

We establish Theorem \ref{1.3} in the setting of $\log$-concave functions and subsequently extend these results to the case of $s$-concave functions for $s>0$. 
\subsection{An auxiliary result}
The following result will be needed in our analysis. The equality case follows from a straightforward computation. 
\begin{lem}\cite[Lemma 2.6]{VD89}
Let $\omega:\left [ 0,\infty  \right )\to \left [ 0,\infty  \right ) $ be decreasing with
\begin{eqnarray}
0<\int_{0}^{\infty }t^{\alpha -1 }\omega \left ( t \right ) dt < \infty 
\end{eqnarray}
for $\alpha >0$. If $\varphi:\left [ 0,\infty  \right )\to \left [ 0,\infty  \right )$ is non-zero, with $\varphi \left ( 0 \right )  =0$, and such that $t\longmapsto \varphi \left ( t \right )$ and $t\longmapsto \varphi \left ( t \right )/t $ are increasing on $(0,\infty)$, then
\begin{eqnarray}
\zeta \left ( \alpha  \right )=\left ( \frac{\int_{0}^{\infty }t^{\alpha-1 }\omega \left ( \varphi \left ( t \right )  \right )dt   }{\int_{0}^{\infty }t^{\alpha-1 }\omega\left (t  \right )    dt   }  \right )^{\frac{1}{\alpha } }
\end{eqnarray}
is a decreasing function of $\alpha$ on $(0,\infty)$. Moreover, $\zeta$ is constant on $(0,\infty)$ if $\varphi \left ( t \right ) =\lambda t$ on $\left [ 0,\infty \right )$ for some $\lambda>0$.  
\end{lem}
\begin{lem}\cite[Lemma 17]{JM22}\label{7bb}
Let $\omega:\left [ 0,\infty  \right )\to \left [ 0,\infty  \right ) $ be decreasing with
\begin{eqnarray}
0<\int_{0}^{\infty }t^{\alpha -1 }\omega \left ( t \right ) dt < \infty 
\end{eqnarray}
for $\alpha >0$ and 
\begin{eqnarray}
0<\int_{0}^{\infty }t^{\alpha -1 }\left ( \omega \left ( 0 \right )- \omega \left ( t \right ) \right )  dt < \infty
\end{eqnarray}
for $-1<\alpha  <0$. If $\varphi:\left [ 0,\infty  \right )\to \left [ 0,\infty  \right )$ is non-zero, with $\varphi \left ( 0 \right )  =0$, and such that $t\longmapsto \varphi \left ( t \right )$ and $t\longmapsto \varphi \left ( t \right )/t $ are increasing on $(0,\infty)$, then
\begin{eqnarray}
\zeta  \left ( \alpha  \right )=\left\{\begin{matrix}
\left ( \frac{\int_{0}^{\infty }t^{\alpha-1 }\omega \left ( \varphi \left ( t \right )  \right )dt   }{\int_{0}^{\infty }t^{\alpha-1 }\omega\left (t  \right )    dt   }  \right )^{\frac{1}{\alpha } }
& for\; \alpha >0\\\exp \left ( \int_{0}^{\infty}\frac{\omega \left ( \varphi \left ( t \right )  \right )-\omega \left ( t \right )  }{t\omega\left ( 0 \right ) } dt  \right ) &for \; \alpha=0\\
\left ( \frac{\int_{0}^{\infty }t^{\alpha-1 }\left ( \omega \left ( \varphi \left ( t \right )  \right )-\omega \left ( 0 \right ) \right )  dt   }{\int_{0}^{\infty }t^{\alpha-1 }\left ( \omega\left (t  \right )-\omega \left ( 0 \right )  \right )     dt   }  \right )^{\frac{1}{\alpha } } &\;\;\;\;\;\;\;\;for\;-1<\alpha <0
\end{matrix}\right.\nonumber
\end{eqnarray}
is a continuous, decreasing function of $\alpha$ on $(-1,\infty)$. Moreover, $\zeta$ is constant on $(-1,\infty)$ if $\varphi \left ( t \right ) =\lambda t$ on $\left [ 0,\infty \right )$ for some $\lambda>0$.  
\end{lem}

\subsection{Results for log-concave functions.}
We begin with an elementary computation, requiring the following notation. For functions $f\in L^{2}\left ( \mathbb{R}^n  \right )$,  $h\in L^{2}\left ( \mathbb{R}^n  \right )$, and for any $y\in \mathbb{R}^n $, we define 
\begin{eqnarray}\label{e7.1}
\mathscr{G}\left ( f,h \right )\left ( y \right )=\int_{\mathbb{R}^n }f\left ( x+y \right )h\left ( x \right )dx.
\end{eqnarray}
The simplex $\bigtriangleup _{n}\subseteq \mathbb{R}^n $ is defined as the convex hull of the origin and the standard basis vectors $e_{1},\dots ,e_{n} $. Let $B^{n} _{1}=\left \{ \left ( x_{1},\dots ,x_{n}   \right ):\left | x_{1}  \right |+ \cdots+\left | x_{n}  \right |\le 1     \right \}$ and $\mathbb{R}^n _{+}=\left \{  \left ( x_{1},\dots ,x_{n}   \right )\in \mathbb{R}^n: x_{1},\dots ,x_{n} \ge 0 \right \} $.
\begin{lem}\cite[Lemma 18]{JM22}\label{7.3}
If $f\left ( x \right )=h\left ( x \right )=e^{-\left \| x \right \|_{\bigtriangleup_{n} }  }$ for $x\in \mathbb{R}^n$ , then
\begin{eqnarray}
\mathscr{G}\left ( f,h \right )\left ( y \right )=\frac{1}{2^{n} }e^{-\left \| y \right \|_{B^{n} _{1} }  }.\notag
\end{eqnarray}
for $y\in \mathbb{R}^n $.
\end{lem}
The following theorem shows set inclusion relation about radial mean bodies of functions.
\begin{thm}\label{5.3}
For non-zero, even, log-concave functions $f$, $h\in L^{2}\left ( \mathbb{R}^n  \right )$, the following inclusion relation holds:
\begin{eqnarray}
\frac{1}{\Gamma \left ( \beta +1 \right )^{\frac{1}{\beta } }  }R_{\beta }\left ( f,h \right )\subseteq \frac{1}{\Gamma \left ( \alpha  +1 \right )^{\frac{1}{\alpha } }  }R_{\alpha  }\left ( f,h \right )\notag
\end{eqnarray}
for $0< \alpha < \beta < \infty$. Equality is attained if $f\left ( x \right )=h\left ( x \right )=ae^{-\left \| x-x_{0}  \right \|_{\bigtriangleup}}$ for $x\in \mathbb{R}^n$, where $a > 0$, $x_{0}\in \mathbb{R}^n$ and $\bigtriangleup $ an $n$-dimensional simplex having a vertex at the origin. 
\end{thm}
\begin{proof}
Since $f$, $h$ are even, the function $ g\left ( y \right ) =\mathscr{G}\left ( f,h \right )\left ( y \right ) $ is even. Furthermore, the log-concavity of $g$ follows immediately from Lemma \ref{2aa}, and consequently $g$ achieves its maximal value at $y=0$ and $g(0)=\int_{\mathbb{R}^n }f\left ( x \right )h\left ( x \right )dx$.

For a fixed $\xi \in \mathbb{S}^{n-1}$, we observe that the function $t\longmapsto g\left ( t\xi  \right )$ is positive, decreasing, log-concave. Using Lemma \ref{7bb} to $\omega \left ( t \right )=g\left ( 0 \right )e^{-t}$, $\varphi \left ( t \right )=-\log\left ( g\left ( t\xi  \right )/g\left ( 0 \right )   \right ) $ and by (\ref{1c}), (\ref{e7.1}) and (\ref{7a}), we have
\begin{eqnarray}
\zeta\left ( \alpha  \right )=\left ( \frac{\int_{0}^{\infty }t^{\alpha -1}g\left ( t\xi  \right )dt   }{\int_{0}^{\infty }t^{\alpha -1}g\left ( 0 \right )e^{-t}dt    }  \right )^{\frac{1}{\alpha } }=\frac{\rho _{R_{\alpha \left ( f,h \right ) } }\left ( \xi  \right )}{\Gamma \left ( \alpha +1 \right )^{\frac{1}{\alpha } }  } \notag
\end{eqnarray}
for $\alpha > 0  $. This establishes the inclution relation by Lemma \ref{7bb}. Equality in Lemma \ref{7bb} occurs precisely when $g\left ( t\xi  \right )=e^{-\lambda\left ( \xi  \right )t } $ for some function $\lambda:{\mathbb{S}}^{n-1}\to\infty$. Consequently, the equality case is obtained form Lemma \ref{7.3} combined with the  volume-preserving $GL\left ( n \right )$ and translation invariance and homogeneity properties of $  \left | R_{\alpha }  \left ( f,h \right )
\right |$.
\end{proof}
\noindent\textbf{Proof of Theorem \ref{1.3}}.
By $0<\alpha <n$ and Theorem \ref{5.3}, we can see that
\begin{eqnarray*}
	\frac{1}{\Gamma \left ( n +1 \right )^{\frac{1}{n } }  }R_{n }\left ( f,h \right )\subseteq \frac{1}{\Gamma \left ( \alpha  +1 \right )^{\frac{1}{\alpha } }  }R_{\alpha  }\left ( f,h \right ),
\end{eqnarray*}
there is a equality in the inequality if $f\left ( x \right )=h\left ( x \right )=ae^{-\left \| x-x_{0}  \right \|_{\bigtriangleup }  }$ for $x\in \mathbb{R}^n  $ with $a \ge 0,x_{0}\in \mathbb{R}^n $ and $\bigtriangleup $ an $n$-dimensional simplex having a vertex at the origin. 
Thus
\begin{eqnarray*}
	\left|\frac{1}{\Gamma \left ( n +1 \right )^{\frac{1}{n } }  }R_{n }\left ( f,h \right )\right|\le \left|\frac{1}{\Gamma \left ( \alpha  +1 \right )^{\frac{1}{\alpha } }  }R_{\alpha  }\left ( f,h \right )\right|.
\end{eqnarray*}
By (\ref{7c}) and (\ref{7a}), we have 
\begin{eqnarray*}
	\frac{1}{\Gamma \left ( n +1 \right)}\frac{{||f||}_{1}{||h||}_{1}}{{\int_{\mathbb{R}^n} }fhdx}
	\le \left|\frac{1}{\Gamma \left ( \alpha  +1 \right )^{\frac{1}{\alpha } }  }\left(\frac{\alpha}{\int_{\mathbb{R}^n} fhdx}\right)^{\frac{1}{\alpha}}{S}_{\alpha }(f,h)\right|.
\end{eqnarray*}
Now raise both sides to the $\frac{\alpha}{n}$-th power to get
\begin{eqnarray*}
	\left ( \left \| f \right \|_{1}\left \| h \right \|_{1}    \right )^{\frac{\alpha }{n} }\left ( \int_{\mathbb{R}^n }f\left ( x \right )h\left ( x \right )dx    \right )^{1-\frac{\alpha }{n} }\le \frac{\Gamma \left ( n+1 \right )^{\frac{\alpha }{n} }  }{\Gamma \left ( \alpha  \right ) }\left | S_{\alpha }\left ( f,h \right )   \right |^{\frac{\alpha }{n} }.
\end{eqnarray*}
We complete the proof of the first inequality of (\ref{1.333}).

For the second inequality of (\ref{1.333}), by H\"older’s inequality and $0<\alpha <n$, we have
\begin{eqnarray*}
	&&\left ( \left \| f \right \|_{1}\left \| h \right \|_{1}    \right )^{\frac{\alpha }{n+\alpha} }\left ( \int_{\mathbb{R}^n }f\left ( x \right )h\left ( x \right )dx    \right )^{\frac{n-\alpha }{n+\alpha} }\nonumber\\
	&=&\left ( {\left \| f \right \|_{1}}^{\frac{2\alpha }{n+\alpha}}\left ( \int_{\mathbb{R}^n }f\left ( x \right )h\left ( x \right )dx    \right )^{\frac{n-\alpha }{n+\alpha} }    \right )^{\frac{1 }{2} }\left ( {\left \| h \right \|_{1}}^{\frac{2\alpha }{n+\alpha}}\left ( \int_{\mathbb{R}^n }f\left ( x \right )h\left ( x \right )dx    \right )^{\frac{n-\alpha }{n+\alpha} }    \right )^{\frac{1 }{2} }\nonumber\\
	&\ge&\left ( \int_{\mathbb{R}^n }  f ^{\frac{2\alpha }{n+\alpha} }( fh )^{\frac{n-\alpha }{n+\alpha} }    dx\right )^{\frac{1 }{2} }\left ( \int_{\mathbb{R}^n }  h ^{\frac{2\alpha }{n+\alpha} }( fh )^{\frac{n-\alpha }{n+\alpha} }    dx\right )^{\frac{1 }{2} }.
\end{eqnarray*} 
Raising both sides to the $\frac{n+\alpha}{n}$-th power to get
\begin{eqnarray*}
	&&\left ( \left \| f \right \|_{1}\left \| h \right \|_{1}    \right )^{\frac{\alpha }{n} }\left ( \int_{\mathbb{R}^n }f\left ( x \right )h\left ( x \right )dx    \right )^{\frac{n-\alpha }{n} }\nonumber\\
	&\ge&\left ( \int_{\mathbb{R}^n }  f ^{\frac{2\alpha }{n+\alpha} }( fh )^{\frac{n-\alpha }{n+\alpha} }    dx\right )^{\frac{n+\alpha }{2n} }\left ( \int_{\mathbb{R}^n }  h ^{\frac{2\alpha }{n+\alpha} }( fh )^{\frac{n-\alpha }{n+\alpha} }    dx\right )^{\frac{n+\alpha }{2n} }\nonumber\\
	&=&\left ( \int_{\mathbb{R}^n }  ( f ^{\frac{n+\alpha }{2n} }h ^{\frac{n-\alpha }{2n} } )^{\frac{2n }{n+\alpha} }    dx\right )^{\frac{n+\alpha }{2n} }\left ( \int_{\mathbb{R}^n }  ( h ^{\frac{n+\alpha }{2n} }f ^{\frac{n-\alpha }{2n} } )^{\frac{2n }{n+\alpha} }    dx\right )^{\frac{n+\alpha }{2n} }\nonumber\\
	&=& {\left \| {f}^{\frac{n+\alpha }{2n}} {h}^{\frac{n-\alpha }{2n}}\right \|}_{\frac{2n }{n+\alpha}}{\left \| {h}^{\frac{n+\alpha }{2n}} {f}^{\frac{n-\alpha }{2n}}\right \|}_{\frac{2n }{n+\alpha}}. 
\end{eqnarray*} 
We complete the proof of the second inequality of the inequalities (\ref{1.333}). We can prove inequalities (\ref{1.12}) following the same method as the proof of inequalities (\ref{1.333}).
\qed

\subsection{Results for s-concave functions}\label{s7.3}
\begin{thm}\label{7cc} 
Let $B(\cdot, \cdot)$ denote Beta function and $s> 0$. For non-zero, even and $s$-concave functions $f$,$h\in L^{2}\left ( \mathbb{R}^n  \right )$, the following inclusion relation holds:
\begin{eqnarray}
	\frac{1}{\left ( \left ( n+\frac{2}{s}   \right )B\left ( \beta +1 ,n+\frac{2}{s} \right )  \right )^{\frac{1}{\beta } }  }R_{\beta  }\left ( f,h \right )\subseteq \frac{1}{\left ( \left ( n+\frac{2}{s}   \right )B\left ( \alpha +1 ,n+\frac{2}{s} \right )  \right )^{\frac{1}{\alpha  } }  }R_{\alpha }\left ( f,h \right ) \notag
\end{eqnarray}
for $0< \alpha < \beta < \infty $.
\end{thm}
\begin{proof}
Since $s$-concave function is log-concave function, as in the proof of Theorem \ref{5.3}, we have that the function $ g\left ( y \right ) =\mathscr{G}\left ( f,h \right )\left ( y \right ) $ is even, continuous, and achieves its maximum at $y=0$. For any $\xi \in \mathbb{S}^{n-1}$, we observe that the function $t\longmapsto g\left ( t\xi  \right )$ is positive and decreasing. Moreover, according to Lemma \ref{2aa}, the function $g$ possesses $r$-concavity where $r=\frac{s}{ns+2}$. By Lemma \ref{7bb} with $\omega \left ( t \right )=g\left ( 0 \right )\left ( 1-rt \right )^{\frac{1}{r} } _{+} $ and $\varphi \left ( t \right )=\left ( 1-\left ( g\left ( t\xi  \right )/g\left ( 0 \right )   \right )^{r}   \right )/r$, and by (\ref{1c}), (\ref{e7.1}) and (\ref{7a}), we have 
\begin{eqnarray}
	\zeta\left ( \alpha  \right )=\left ( \frac{\int_{0}^{\infty }t^{\alpha -1}g\left ( t\xi  \right )dt   }{\int_{0}^{\infty }g\left ( 0 \right ) t^{\alpha -1}\left ( 1-rt \right )^{\frac{1}{r} } _{+}dt    }  \right )^{\frac{1}{\alpha } }=\frac{r\rho _{R_{\alpha \left ( f,h \right ) } }\left ( \xi  \right )} { \left ( \alpha B \left ( \alpha ,1+\frac{1}{r}  \right ) \right )^{\frac{1}{\alpha } }    } \notag
\end{eqnarray}
for $\alpha > 0$. The conclusion follows directly from Lemma \ref{7bb}.
\end{proof}
For the limit as $s\to 0$, Theorem \ref{7cc} reduces to Theorem \ref{5.3}. In the case where $s\to \infty $
and and $f=1_{E}$, $h=1_{F}$ are indicator functions of convex bodies $F\subseteq E\subseteq \mathbb{R}^n$, Theorem \ref{7cc} yields the inclusion
\begin{eqnarray}\label{30}
\frac{1}{\left (  n   B\left ( \beta +1 ,n \right )  \right )^{\frac{1}{\beta } }  }R_{\beta  }\left ( E,F \right )\subseteq \frac{1}{\left (  n   B\left ( \alpha +1 ,n\right )  \right )^{\frac{1}{\alpha  } }
}R_{\alpha }\left ( E,F\right )
\end{eqnarray}
for $0< \alpha < \beta $. When $F=E$, this result coincides with Theorem 5.5 of R. Gardner and G. Zhang \cite{GZ98}, who established that equality in (\ref{30}) holds for $n$-dimensional simplices. 
From the proof of Theorem \ref{7cc}, under the stated assumptions, we deduce that
\begin{eqnarray}\label{31}
\frac{\rho _{R_{\alpha }\left ( f,h \right ) }\left ( \xi  \right ) }{\left ( \left ( n+\frac{2}{s}   \right )B\left ( \beta +1 ,n+\frac{2}{s} \right )  \right )^{\frac{1}{\beta } }  }\le \frac{\rho _{R_{\alpha }\left ( f,h \right )}\left ( \xi  \right )}{\left ( \left ( n+\frac{2}{s}   \right )B\left ( \alpha +1 ,n+\frac{2}{s} \right )  \right )^{\frac{1}{\alpha  } }  }
\end{eqnarray}
with equality for $\xi \in \mathbb{S}^{n-1} $ when $f\left ( x \right )=h\left ( x \right ) $ and $\mathscr{G}\left ( f,h \right )\left ( t\xi  \right )=a\left ( 1-\lambda t \right )^{\frac{1}{s} } _{+}  $ for $t> 0$, where $a\ge 0$ and $\lambda > 0$. The subsequent lemma demonstrates that the inequality in (\ref{31}) is sharp in certain directions and the optimality of the constants in Theorem \ref{7cc}.
\begin{lem}\cite[Lemma 22]{JM22}\label{5.5}
For $s>0$, suppose $f\left ( x \right ) =h\left ( x \right )=\left ( 1-\left \| x   \right \|_{\bigtriangleup _{n} }  \right )^{\frac{1}{s} } _{+}$ for $x\in \mathbb{R}^n$. Then, the function $\mathscr{G}\left ( f,h \right )$ satisfies
\begin{eqnarray}
	\mathscr{G}\left ( f,h \right )\left ( y \right )=a\left ( 1-\frac{1}{2} \left \| y \right \|_{B^{n} _{1}  }
	\right )^{n+\frac{2}{s} } 	\notag
\end{eqnarray}
for all $y\in \mathbb{R}^n$ with $y_{1}+\cdots +y_{n}=0$, where the constant $a$ is given by $a=B\left ( n,1+\frac{2}{s}  \right )/\left ( n-1 \right )!$.
\end{lem}
\begin{cor}
For $s>0$, suppose $f$, $h\in L^{2}\left ( \mathbb{R}^n  \right )$ are even and s-concave on their supports. Then the following inequalities hold: for $0< \alpha < n$
\begin{eqnarray}\label{ee7.1}
	\frac{\left ( nB\left ( n,n+1+\frac{2}{s}  \right )  \right )^{\frac{\alpha }{n} }  }{B\left ( \alpha ,n+1+\frac{2}{s} \right ) } \left | S_{\alpha }\left ( f,h \right )   \right |^{\frac{\alpha }{n} } &\ge&  \left ( \left \| f \right \|_{1}\left \| h \right \|_{1}    \right )^{\frac{\alpha }{n} }\left ( \int_{\mathbb{R}^n }f\left ( x \right )h\left ( x \right )dx    \right )^{1-\frac{\alpha }{n} }\nonumber\\
	&\ge& {\left \| {f}^{\frac{n+\alpha }{2n}} {h}^{\frac{n-\alpha }{2n}}\right \|}_{\frac{2n }{n+\alpha}}{\left \| {h}^{\frac{n+\alpha }{2n}} {f}^{\frac{n-\alpha }{2n}}\right \|}_{\frac{2n }{n+\alpha}},
		\end{eqnarray}
	and for $\alpha >n $,
	\begin{eqnarray}\label{ee7.2}
		\frac{\left ( nB\left ( n,n+1+\frac{2}{s}  \right )  \right )^{\frac{\alpha }{n} }  }{B\left ( \alpha ,n+1+\frac{2}{s} \right ) } \left | S_{\alpha }\left ( f,h \right )   \right |^{\frac{\alpha }{n} } &\leq &  \left ( \left \| f \right \|_{1}\left \| h \right \|_{1}    \right )^{\frac{\alpha }{n} }\left ( \int_{\mathbb{R}^n }f\left ( x \right )h\left ( x \right )dx    \right )^{1-\frac{\alpha }{n} }\nonumber\\
		&\leq & {\left \| {f}^{\frac{n+\alpha }{2n}} {h}^{\frac{n-\alpha }{2n}}\right \|}_{\frac{2n }{n+\alpha}}{\left \| {h}^{\frac{n+\alpha }{2n}} {f}^{\frac{n-\alpha }{2n}}\right \|}_{\frac{2n }{n+\alpha}}.
	\end{eqnarray}
	\end{cor}
	\begin{proof}
		By $0<\alpha <n$ and Theorem \ref{7cc}, we can see that
		\begin{eqnarray*}
			\frac{1}{\left ( \left ( n+\frac{2}{s}   \right )B\left ( n +1 ,n+\frac{2}{s} \right )  \right )^{\frac{1}{n } }  }R_{n  }\left ( f,h \right )\subseteq \frac{1}{\left ( \left ( n+\frac{2}{s}   \right )B\left ( \alpha +1 ,n+\frac{2}{s} \right )  \right )^{\frac{1}{\alpha  } }  }R_{\alpha }\left ( f,h \right ) \notag.
		\end{eqnarray*} 
		Thus
		\begin{eqnarray*}
			\frac{1}{\left  ( n+\frac{2}{s}   \right )B\left ( n +1 ,n+\frac{2}{s} \right )   }\left|R_{n  }\left ( f,h \right )\right|\le \frac{1}{\left ( \left ( n+\frac{2}{s}   \right )B\left ( \alpha +1 ,n+\frac{2}{s} \right )  \right )^{\frac{n}{\alpha  } }  }\left|R_{\alpha }\left ( f,h \right )\right|.
		\end{eqnarray*}
		Therefore, by (\ref{7c}) and (\ref{7a}), we have 
		\begin{eqnarray*}
			&&\frac{1}{ \left ( n+\frac{2}{s}   \right )B\left ( n +1 ,n+\frac{2}{s} \right )    }\frac{{||f||}_{1}{||h||}_{1}}{{\int_{\mathbb{R}^n} }fhdx}\nonumber\\
			&\le& \left|\frac{1}{\left ( \left ( n+\frac{2}{s}   \right )B\left ( \alpha +1 ,n+\frac{2}{s} \right )  \right )^{\frac{1}{\alpha  } }  }\left(\frac{\alpha}{\int_{\mathbb{R}^n} fhdx}\right)^{\frac{1}{\alpha}}{S}_{\alpha }(f,h)\right|,
		\end{eqnarray*}
		Now raise both sides to the $\frac{\alpha}{n}$-th power and by $pB(p,q+1)=qB(p+1,q)$, we have
		\begin{eqnarray*}
			\left ( \left \| f \right \|_{1}\left \| h \right \|_{1}    \right )^{\frac{\alpha }{n} }\left ( \int_{\mathbb{R}^n }f\left ( x \right )h\left ( x \right )dx    \right )^{1-\frac{\alpha }{n} }\le \frac{\left ( nB\left ( n,n+1+\frac{2}{s}  \right )  \right )^{\frac{\alpha }{n} }  }{B\left ( \alpha ,n+1+\frac{2}{s} \right ) } \left | S_{\alpha }\left ( f,h \right )   \right |^{\frac{\alpha }{n} }.
		\end{eqnarray*}
		We complete the proof of the first inequality of (\ref{ee7.1}).
		Since the second inequality of inequalities (\ref{ee7.1}) is the same as the second inequality of inequalities (\ref{1.333}), we omit its proof. We can prove inequalities (\ref{ee7.2}) following the same method as the proof of inequalities (\ref{ee7.1}).
	\end{proof}

\section{Competing interests}
No competing interest is declared.

\section{Acknowledgments}
This work is supported in part by funds from the National Natural Science Foundation of China NSFC 12371137 and NSFC 11971080.




\vspace{1.5ex}

\noindent Youjiang Lin\\
\noindent\small  School of Mathematics and Statistics, Chongqing Technology and Business University, Chongqing 400067, PR China\\
\noindent\small{ School of Mathematical Sciences, Hebei Normal University, Shijiazhuang, Hebei,
050024, PR China} 
E-mail address: yjl@ctbu.edu.cn,yjlin@hebtu.edu.cn\\

\noindent Jinghong Zhou, Jiaming Lan\\
\noindent\small  School of Mathematics and Statistics, Chongqing Technology and Business University, Chongqing 400067, PR China\\
\noindent\small{ zjh@ctbu.edu.cn, lanjiaming@ctbu.edu.cn}\\
\end{document}